\documentclass[12pt,a4]{amsart}
\usepackage{amsmath}
\usepackage{amssymb}
\usepackage{amsthm}
\usepackage{enumitem}
\usepackage{url}
\usepackage{soul}
\usepackage{times}
\usepackage[T1]{fontenc}

\usepackage{color}
\usepackage{dsfont}
\usepackage{epsfig}
\usepackage[hidelinks]{hyperref}
\usepackage{setspace}
\usepackage{epstopdf}
\usepackage[authoryear,round,sort]{natbib}
\usepackage{comment}
\usepackage{booktabs}
\usepackage{float}
\usepackage{tikz}
\usepackage{algorithm}
\usepackage{stmaryrd}
\usepackage{subcaption,graphicx}
\usepackage{float}
\numberwithin{equation}{section}

\usepackage[font=small,labelfont=bf]{caption}
\DeclareCaptionFont{tiny}{\tiny}
\usepackage{geometry}
\theoremstyle{plain}
\newtheorem{proposition}{Proposition}[section]

\newtheorem{theorem}{Theorem}[section]
\newtheorem{lemma}{Lemma}[section]
\newtheorem{corollary}{Corollary}[section]

\theoremstyle{definition}
\newtheorem{definition}{Definition}[section]
\newtheorem{assumption}{Assumption}[section]
\newtheorem{example}{Example}[section]

\theoremstyle{remark}
\newtheorem{rk}{Remark}[section]
\expandafter\let\expandafter\oldproof\csname\string\proof\endcsname
\let\oldendproof\endproof
\renewenvironment{proof}[1][\proofname]{%
  \oldproof[\noindent\textbf{#1.} ]%
}{\oldendproof}
\newcommand{\1}{\mathds{1}}

\newcommand{\E}{\mathbb{E}}

\newcommand{\be}{\begin{equation}}
\newcommand{\ee}{\end{equation}}
\newcommand{\by}{\begin{eqnarray*}}
\newcommand{\ey}{\end{eqnarray*}}

\renewcommand{\leq}{\leqslant}
\renewcommand{\geq}{\geqslant}
\usepackage{xcolor}
\definecolor{dark-red}{rgb}{0.4,0.15,0.15}
\definecolor{dark-blue}{rgb}{0.15,0.15,0.4}
\definecolor{medium-blue}{rgb}{0,0,0.5}
\hypersetup{
    colorlinks, linkcolor={red},
    citecolor={blue}, urlcolor={blue}
}

\begin{document}
\title{Metropolis-Hastings reversiblizations of non-reversible Markov chains}
\author{Michael C.H. Choi}
\address{Institute for Data and Decision Analytics, The Chinese University of Hong Kong, Shenzhen, Guangdong, 518172, P.R. China.}
\email{michaelchoi@cuhk.edu.cn}
\date{\today}
\maketitle


\begin{abstract}
	We study two types of Metropolis-Hastings (MH) reversiblizations for non-reversible Markov chains with Markov kernel $P$. While the first type is the classical Metropolised version of $P$, we introduce a new self-adjoint kernel which captures the opposite transition effect of the first type, that we call the second MH kernel. We investigate the spectral relationship between $P$ and the two MH kernels. Along the way, we state a version of Weyl's inequality for the spectral gap of $P$ (and hence its additive reversiblization), as well as an expansion of $P$. Both results are expressed in terms of the spectrum of the two MH kernels. In the spirit of \cite{Fill91} and \cite{Paulin15}, we define a new pseudo-spectral gap based on the two MH kernels, and show that the total variation distance from stationarity can be bounded by this gap. We give variance bounds of the Markov chain in terms of the proposed gap, and offer spectral bounds in metastability and Cheeger's inequality in terms of the two MH kernels by comparison of Dirichlet form and Peskun ordering.
	\smallskip
	
	\noindent \textbf{AMS 2010 subject classifications}: Primary 60J05, 60J10; Secondary 60J20, 37A25, 37A30
	
	\noindent \textbf{Keywords}: non-reversible Markov chain; spectral gap; Metropolis-Hastings algorithm; mixing time; Weyl's inequality; variance bounds
\end{abstract}

\tableofcontents
\allowdisplaybreaks

\section{Introduction}

Consider a Markov chain with Markov kernel $P$ and stationary distribution $\pi$ with its time-reversal $P^*$ on a general state space $\mathcal{X}$. The quantitative rate of convergence to equilibrium is well-known to be closely connected to the spectrum or the spectral gap of $P$, see for instance \cite{AF14, LSC97, LPW09, MT06, MT09} and the references therein. In the reversible case, that is, when $P$ is viewed as a linear self-adjoint operator in $L^2(\pi)$, \cite{RR97} shows that the existence of an $L^2$-spectral gap is equivalent to $P$ being geometrically ergodic. The main technical insight relies heavily on the spectral theory of self-adjoint operators, which facilitates the analysis of the spectrum of $P$.

However, $P$ need not be reversible in general. If $P$ is non-reversible, the analysis on the rate of convergence is fragmentally understood, possibly due to a much less developed spectral theory for non-self-adjoint operators. We now describe three different approaches that have been elaborated to overcome this difficulty.

The first approach, initiated by \cite{Fill91}, is to resort to an appropriate \textit{reversiblized} version of $P$ and analyze how its spectrum can be related to the chi-squared distance to stationarity of the original chain. Two reversiblizations are proposed, namely the multiplicative reversiblization $PP^*$ and the additive reversiblization $(P+P^*)/2$. In the discrete-time setting, it is shown that the second largest eigenvalue of $PP^*$ can be used to upper bound the distance from stationarity, while the spectral gap of $(P+P^*)/2$ is used for continuous-time Markov chain. More recently, \cite{Paulin15} generalizes this approach and defines a \textit{pseudo}-spectral gap, based upon the maximum spectral gap of $P^{*k}P^k$ for $k \geq 1$. He demonstrates that the proposed gap plays a similar role as that of spectral gap in the reversible case. He proves variance bounds and Bernstein inequality based on his proposed gap. In this vein, we also mention the work of \cite{J13} who gives mixing time bounds for general non-reversible Markov chains in terms of absolute spectral gap.

The second approach, proposed by \cite{KM12}, is to cast $P$ in a weighted Banach space $L^{\infty}_V$ instead of the classical $L^2(\pi)$ framework, where $V$ is a Lyapunov function associated with $P$. In particular, they show that for a $\phi$-irreducible and aperiodic Markov chain, $P$ is geometrically ergodic if and only if $P$ admits a spectral gap in the space $L^{\infty}_V$ equipped with the $V$-norm. They also give an example in which a non-reversible Markov chain is geometrically ergodic yet it fails to have a $L^2(\pi)$ spectral gap.

The third approach, initiated by \cite{Patie-Savov} and \cite{Miclo2016}, is to resort to  intertwining relationship, to build a link between the non-reversible and reversible chains. \cite{Patie-Savov} investigated the rate of convergence to equilibrium of the generalized Laguerre semigroup and the hypocoercivity phenomenon. Recently, by means of the concept of similarity, \cite{Choi-Patie} investigated the rate of convergence to equilibrium of skip-free chains and the cut-off phenomenon.

The path taken in this paper is in the spirit of the first approach and stems on an additional reversiblization procedure. More specifically, we use and develop further  the celebrated Metropolis-Hastings (MH) algorithm   to provide an original in-depth analysis of non-reversible chains. The aim is to investigate Metropolis-Hastings (MH) reversiblizations, and how it helps to analyze non-reversible chains. The MH algorithm, developed by \cite{M53} and \cite{H70}, is a Markov Chain Monte Carlo method that is of fundamental importance in statistics and other applications, see e.g. \cite{RR04} and the references therein. The idea is to construct from a proposal kernel a reversible chain which converges to a desired distribution. Much of the literature focuses on the speed of convergence of \textit{specific} algorithms, where the proposal kernel (e.g.~a random walk proposal or an Ornstein-Uhlenbeck proposal) is often by itself reversible and the target density is in general \textit{not} the proposal stationary measure. For example, \cite{RT96} investigates the random walk MH with exponential family target density. \cite{HSV14} compares the theoretical performance of random walk MH and pCN algorithm with target density given by their equation $1.1,1.2$ by establishing their Wasserstein spectral gap.

The notion of MH reversiblizations to study non-reversible chains is not entirely new. To the best of our knowledge, this term is first formally introduced by \cite{AF14}, although they did not provide a detailed analysis. Our contributions can be summarized as follows:
\begin{enumerate}
	\item We start by studying two types of MH reversiblizations. The first MH kernel is the classical Metropolis chain of $P$, and we identify a new self-adjoint yet possibly non-Markovian operator that we call the second MH kernel. It captures the opposite transition effect of the first kernel, and thus it can be interpreted as the \emph{dual} in a broad sense. We show that the linear operator $P + P^*$ can be written as the sum of the two MH kernels, which allows us to state a version of Weyl's inequality for the spectral gap of $P$ and its additive reversiblization in the finite state space case. We prove that our bound is sharp by investigating in detail the asymmetric simple random walk on the $n$-cycle. We also give  a spectral-type expansion of $P$ expressed in terms of the spectral measures of the two MH kernels, which we call a \textit{MH pseudospectral} expansion, in terms of the spectral measures of the two MH kernels.
	
	\item We proceed by defining a \textit{pseudo}-spectral gap, that we call the MH-spectral gap, based on the spectrum of the two MH kernels, along the line of work by \cite{Paulin15}. We show that the existence of a MH-spectral gap implies that $P$ is geometrically ergodic. We carry out some numerical examples that reveal that our MH-spectral gap is, for non-reversible chains, a better estimate than the existing bounds found in the literature. Variance bounds are also proved in terms of the proposed gap. Finally, we revisit the notion of metastability and the Cheeger's inequality, to offer a variant of these celebrated inequalities by means of comparison of the non-reversible chain and the two MH kernels.
\end{enumerate}

The rest of the paper is organized as follows. We fix the notation and give a review of the theory of general state space Markov chains as well as the MH algorithm in Section \ref{sec:prelim}. We begin Section \ref{sec:MHr} by formally defining the two MH kernels and state some elementary results, followed by comparing $P$ and the two MH kernels using the Peskun ordering, and we end this section by stating Weyl's inequality for the spectral gap of $P$. Section \ref{sec:pseexp} describes the pseudospectral expansion of $P$. The MH-spectral gap is defined in Section \ref{sec:geomergodMHspecgap}, and we give a number of results that relate Weyl's inequality, geometric ergodicity, mixing time and the MH-spectral gap. Finally, we state the results about variance bounds in terms of the MH-spectral gap in Section \ref{sec:varbd}, and discuss metastability and the Cheeger's inequality bounds in Section \ref{sec:metastable}.

\section{Preliminaries}\label{sec:prelim}
In this section, we review several fundamental notions for Markov chains on a general state space.

Let $X = (X_n)_{n \in \mathbb{N}_0}$ be a time-homogeneous Markov chain on a measurable state space $(\mathcal{X},\mathcal{F})$, and as usual we write $P$ to be the Markov kernel which describes the one-step transition. Recall that for $P : \mathcal{X} \times \mathcal{F} \to [0,1]$ to be a Markov kernel, for each fixed $A \in \mathcal{F}$, the mapping $x \mapsto P(x,A)$ is $\mathcal{F}$-measurable and for each fixed $x \in \mathcal{X}$, the function $A \mapsto P(x,A)$ is a probability measure on $\mathcal{X}$. As we need to handle possibly non-Markovian general kernels in the sequel, we say that a general kernel $K$ acts on a given function $f: \mathcal{X} \to \mathbb{C}$ from the left and a signed measure $\mu$ on $(\mathcal{X},\mathcal{F})$ from the right by
$$ Kf(x) := \int_{\mathcal{X}} f(y) K(x,dy), \quad \mu K(A) := \int_{\mathcal{X}} K(x,A) \mu(dx), \quad x \in \mathcal{X}, A \in \mathcal{F},$$
whenever the integrals exist. 

We say that $\pi$ is a stationary distribution of $X$  if $\pi$ is a probability measure on $(\mathcal{X},\mathcal{F})$ and
$$\int_{\mathcal{X}} P(x,A) \, \pi(dx) = \pi(A), \quad A \in \mathcal{F}.$$
A closely related notion is \textit{reversibility}. We say that $X$ is reversible if there is a probability measure $\pi$ on $(\mathcal{X},\mathcal{F})$ such that the \textit{detailed balance} relation is satisfied:
$$\pi(dx)P(x,dy) = \pi(dy) P(y,dx).$$
Note that detailed balance means the two probability measures are identical on the product space $(\mathcal{X} \times \mathcal{X},\mathcal{F} \times \mathcal{F})$. It is known that if $\pi$ is a reversible probability measure, then $\pi$ is a stationary distribution, yet the converse is not true. Let $L^2(\pi)$ be the Hilbert space of complex valued measurable functions on $\mathcal{X}$ that are squared-integrable with respect to $\pi$, endowed with the inner product $\langle f,g \rangle_{\pi} := \int f \overline{g} \, d\pi$ and the norm $||f||_{\pi} := \langle f,f \rangle_{\pi}^{1/2}$, where the overline $\overline{g}$ is the complex conjugate of $g$. $P$ can then be viewed as a linear operator on $L^2(\pi)$, in which we still denote the operator by $P$. The operator norm of $P$ on $L^2(\pi)$ is
$$||P||_{L^2 \to L^2} = \sup_{\substack{f \in L^2(\pi) \\ ||f||_{\pi}=1}} ||Pf||_{\pi}.$$
Let $P^*$ be the adjoint or time-reversal of $P$ on $L^2(\pi)$, and it can be checked that
$$\pi(dx)P^*(x,dy) = \pi(dy) P(y,dx).$$
This shows that reversibility is equivalent to self-adjointness of $P$. Write $\sigma(P) = \sigma(P|L^2)$ to be the spectrum of $P$ on $L^2(\pi)$, i.e.
$$\sigma(P|L^2) = \{\lambda \in \mathbb{C}\backslash 0 : (\lambda I - P) \, \text{does not have a bounded inverse} \}.$$

If $P$ is self-adjoint, then $\sigma(P) \subseteq [-1,1]$. In addition, the spectral theorem for self-adjoint operators gives
\begin{align}\label{eq:specthm}
	P = \int_{\sigma(P)} \lambda \, \mathcal{E}(d \lambda),
\end{align}
where $\mathcal{E}$ is the spectral measure associated with $P$. When considering the spectral gap of $P$, it is often convenient for us to restrict to the space $L^2_0(\pi) = \{f \in L^2(\pi) : \E_{\pi} f = 0\}$. We formally define the meaning of an $L^2$-spectral gap.
\begin{definition}[$L^2$-spectral gap]\label{def:L2speg}
	Suppose that $P$ is a Markov kernel with stationary measure $\pi$. If
	$$\beta = ||P||_{L^2_0 \to L^2_0} < 1,$$
	then the (absolute) $L^2$-spectral gap is $\gamma^* = \gamma^*(P) = 1 - \beta$.
\end{definition}
Let
$$\lambda = \lambda(P) := \inf \{ \alpha : \alpha \in \sigma(P|L^2_0) \}, \quad \Lambda = \Lambda(P) := \sup \{ \alpha : \alpha \in \sigma(P|L^2_0) \}.$$
If $P$ is reversible with respect to $\pi$, then it is known that (see, e.g. \cite{Rudolf12})
\begin{align}\label{eq:luspec}
	\lambda = \inf_{\substack{f \in L^2_0(\pi) \\ ||f||_{\pi}=1}} \langle Pf,f \rangle_{\pi}, \quad \Lambda = \sup_{\substack{f \in L^2_0(\pi) \\ ||f||_{\pi}=1}} \langle Pf,f \rangle_{\pi},
\end{align}
$$||P||_{L^2_0 \to L^2_0} = \sup_{\substack{f \in L^2_0(\pi) \\ ||f||_{\pi}=1}} |\langle Pf,f \rangle_{\pi}|.$$
This allows us to deduce that
\begin{align} \label{eq:specgap_Prev}
	\beta = \max\{|\lambda|, \Lambda \}.
\end{align}

We also define the (right) spectral gap for a Markov kernel.

\begin{definition}[spectral gap]\label{def:spectralgap}
	Suppose that $P$ is a Markov kernel with stationary measure $\pi$. The (right) spectral gap is defined to be
	$$\gamma = \gamma(P) := 1 - \sup \{ \mathrm{Re}(\alpha) : \alpha \in \sigma(P|L^2_0) \}.$$
\end{definition}

If $P$ is reversible, then $\gamma = 1 - \Lambda(P)$. In the finite state space setting, it can be shown that $\gamma(P) = \gamma((P+P^*)/2)$ (see e.g. \cite[discussion after Definition $2.1.3$, where we can consider $\lambda$ therein to be our $\gamma(P)$]{LSC97}), so for a general $P$ on a finite state space, $$\gamma(P) = 1 - \Lambda((P+P^*)/2).$$
\begin{rk}
	We recall that in \cite{Fill91}, additive reversiblization $(P+P^*)/2$ and multiplicative reversiblization $PP^*$ are proposed to study mixing for non-reversible chains. In the discrete-time setting, the upper bound involves $\gamma(PP^*)$, while for continuous-time Markov chains, the upper bound depends on $\gamma((P+P^*)/2)$.
\end{rk}

\begin{rk}
	In \cite{Paulin15}, a \textit{pseudo}-spectral gap based on the spectral gap of $P^{*k}P^k$ for $k \geq 1$ is introduced. Precisely, we define
	$$\gamma^{ps} = \gamma^{ps}(P) := \max_{k \geq 1} \{\gamma(P^{*k}P^k)/k\}.$$
\end{rk}

\subsection{The Metropolis-Hastings kernel}

Let $\pi$ be a probability measure on $(\mathcal{X},\mathcal{F})$ that is absolutely continuous with density $\pi$ (with a slight abuse of notation, the density is still denoted by $\pi$) with respect to a reference measure $\mu$ on $\mathcal{X}$, that is, $\pi(dx) = \pi(x) \mu(dx)$. Denote $Q$ to be any Markov kernel on $\mathcal{X}$, where $Q(x,\cdot)$ is absolutely continuous with density $q$ with respect to $\mu$. $\pi$ is the so-called \textit{target} distribution, while $Q$ is commonly known as the \textit{proposal} kernel. Define the \textit{acceptance} probabilities $\alpha(x,y)$ by
\[\alpha(x,y) = \begin{cases} \min \left(\dfrac{\pi(y)q(y,x)}{\pi(x)q(x,y)},1 \right) &\mbox{if } \pi(x)q(x,y) > 0, \\
0 & \mbox{otherwise.} \end{cases}\]
Let $p(x,y) := \alpha(x,y) q(x,y)$, and define the \textit{reject} probabilities $r : \mathcal{X} \to [0,1]$ via $r(x) := 1 - \int p(x,y) \mu(dy)$. The Metropolis-Hastings kernel $P$ is given by
$$P(x,dy) = p(x,y) \mu(dy) + r(x)\delta_x(dy),$$
where $\delta_x$ is the point mass at $x$.

The MH kernel allows for the following algorithmic interpretation. First, we choose $X_0$ and given the current state $X_n$, we generate the \textit{proposal} $Y_{n+1}$ by $Q(X_n,\cdot)$. With probability $\alpha(X_n,Y_{n+1})$, we accept the proposal and set $X_{n+1} = Y_{n+1}$. Otherwise, we reject $Y_{n+1}$ and set $X_{n+1} = X_n$. Finally, we set $n = n+1$ and the above procedure is repeated.

Markov Chain Monte Carlo (MCMC) methods, such as the classical Metropolis-Hastings algorithm, involve constructing a Markov chain which converges to a desired stationary distribution $\pi$ that one would like to sample from. It differs from Monte Carlo methods in the sense that $\pi$ is often difficult to simulate directly, and is particularly useful in situations where we only know $\pi$ up to normalization constants. As described in \cite{RR04}, we can see that the choice of the proposal kernel $Q$ has an significant impact on the performance of the MH algorithm. Common choices of $Q$ includes the symmetric MH ($q(x,y) = q(y,x)$), random walk MH ($q(x,y) = q(y-x)$) and independence MH ($q(x,y) = q(y)$).
		
\section{Metropolis-Hastings reversiblizations}\label{sec:MHr}

From now on, unless otherwise stated, we assume $X$ is a $\phi$-irreducible Markov chain, which may not be reversible, with Markov kernel $P$ and stationary distribution $\pi$. We also assume that $P(x,\cdot)$, $P^*(x,\cdot)$ and $\pi$ share a common dominating reference measure $\mu$ on $(\mathcal{X},\mathcal{F})$, with density denoted by $p(x,\cdot)$, $p^*(x,\cdot)$ and $\pi(\cdot)$, respectively. Furthermore, we assume that the set $\{x : \pi(x) = 0\}$ is a $\mu$-null set.

Given a Markov chain $X$ with Markov kernel $P$ and stationary distribution $\pi$, we can obtain a \textit{MH}-\textit{reversiblized} chain by taking the proposal kernel to be $P$. The resulting process is what we called \textit{the first MH chain}.

\begin{definition}[The first MH kernel]\label{def:M1}
	The first MH chain, with Markov kernel denoted by $M_1 := M_1(P)$, is the MH kernel with proposal kernel $P$ and target distribution $\pi$. That is, let
	\begin{align*}
	\alpha_1(x,y) &= \begin{cases} \min \left(\dfrac{\pi(y)p(y,x)}{\pi(x)p(x,y)},1 \right) &\mbox{if } \pi(x)p(x,y) > 0, \\
	0 & \mbox{otherwise,} \end{cases} \\
	m_1(x,y) &= \alpha_1(x,y) p(x,y) = \min\{p^*(x,y),p(x,y)\}, \\
	r_1(x) &= \int_{y \neq x} (1-\alpha_1(x,y)) p(x,y) \, \mu(dy) ,
	\end{align*}
	then $M_1$ is given by
	$$M_1(x,dy) = m_1(x,y) \mu(dy) + r_1(x)\delta_x(dy).$$
\end{definition}

By taking a closer look at $m_1$, we can see that the first MH chain \textit{weakens} the transition to $A_x := \{y \in \mathcal{X} : \alpha_1(x,y) < 1\}$, and \textit{follows} the same transition as the original chain $X$ for $A_x^c = \{y : \alpha_1(x,y) = 1\}$. This motivates us to develop what we call \textit{the second MH kernel} $M_2 := M_2(P)$ with density $m_2$, which captures the \textit{opposite} transition of $M_1$. Precisely, we would like to have
$$m_2(x,y) = \begin{cases} p(x,y) &\mbox{if } y \in A_x, \\
p^*(x,y) &\mbox{if } y \in A_x^c \backslash \{x\}. \end{cases} = \max\{p^*(x,y),p(x,y)\}.$$
As a result, we obtain the following:
\begin{definition}[The second MH kernel]\label{def:M2}
	The second MH kernel $M_2 := M_2(P)$ and density $m_2$ are given by
	\begin{align*}
	m_2(x,y) &= \max\{p^*(x,y),p(x,y)\}, \\
	M_2(x,dy) &= m_2(x,y) \mu(dy) - r_1(x)\delta_x(dy).
	\end{align*}
\end{definition}

Note that $M_2$ in general may not be a Markov kernel, since there is no guarantee that $M_2(x,\{x\}) = P(x,\{x\}) - r_1(x) \geq 0$. For instance, if $P$ is the Markov kernel of a finite Markov chain with $P(x,x) = 0$ for all $x \in \mathcal{X}$, then $M_2(x,x) = -r_1(x) \leq 0$. In the following we collect a few elementary properties of $M_1$ and $M_2$.

\begin{lemma}\label{lem:M1M2prop}
	Suppose that $P$ is a Markov kernel with stationary measure $\pi$, with $M_1$ and $M_2$ being the first and second MH kernel of $P$ respectively. Then the following holds.
	\begin{enumerate}[label={\upshape(\roman*)}, align=left, widest=iii, leftmargin=*]
		\item $P + P^* = M_1 + M_2$. In particular, $M_2(x, \mathcal{X}) = 1$, for $x \in \mathcal{X}$. \label{state:1}
		\item $M_1$ and $M_2$ are self-adjoint operators on $L^2(\pi)$. \label{state:2}
		\item $M_1 = M_2 = P$ if and only if $P$ is reversible with respect to $\pi$. \label{state:3}
		\item $M_i(P) = M_i(P^*)$ for $i = 1,2$. \label{state:4}
	\end{enumerate}
\end{lemma}

\begin{proof}
	\begin{enumerate}[label={\upshape(\roman*)}, align=left, widest=iii, leftmargin=*]
		\item This can easily be seen from Definition \ref{def:M1} and \ref{def:M2} together with $$p(x,y) + p^*(x,y) = \min\{p^*(x,y),p(x,y)\} + \max\{p^*(x,y),p(x,y)\}.$$
		
		\item It is well known that $M_1$ is reversible (and hence invariant) w.r.t.~$\pi$, see e.g.~\cite[Proposition $2$]{RR04}. To see that $M_2$ is a self-adjoint operator on $L^2(\pi)$, we use \ref{state:1}. That is,
		$$M_2^* = (P + P^* - M_1)^* = P^* + P - M_1 = M_2.$$
		
		\item If $M_1 = M_2 = P$, then \ref{state:1} gives $P^* = M_1 + M_2 - P = P$. Conversely, when $P$ is reversible w.r.t.~$\pi$, then $\alpha(x,y) = 1$ $\mu \times \mu \, a.e.$, and hence $M_1 = P$ by Definition \ref{def:M1}. It follows again from \ref{state:1} that $M_2 = P + P^* - M_1 = P$.
		
		\item Using the fact that $P^{**} = P$ and Definition \ref{def:M1}, $M_1(P) = M_1(P^*).$ Next, \ref{state:1} gives
		$$M_2(P) = P + P^* - M_1(P) = P^* + P - M_1(P^*) = M_2(P^*).$$
	\end{enumerate}
\end{proof}

\begin{rk}
	As remarked earlier, although $M_2$ is not a Markov kernel in general, it is a self-adjoint operator in $L^2(\pi)$ and satisfies $\pi M_2 = \pi (P + P^* - M_1) = \pi$.
\end{rk}


\subsection{Peskun Ordering}
We aim to investigate some further relationships and properties of the spectra of $P, M_1$ and $M_2$ via the so-called Peskun ordering, which was first introduced by \cite{Pesk73} as a partial ordering for Markov kernels on finite state space, and was further generalized by \cite{Tie98} to general state space.

\begin{definition}[Peskun Ordering]\label{def:peksun}
	Suppose that $P_1, P_2$ are general kernels with invariant distribution $\pi$. $P_1$ dominates $P_2$ off the diagonal, written as $P_1 \succeq P_2$, if for $\pi$-almost all $x$, $P_1(x,A) \geq P_2(x,A)$ for all $A \in \mathcal{F}$ with $x \notin A$.
\end{definition}

Note that we are not restricting to Markov kernels in Definition \ref{def:peksun}, since $M_2$ in general may not be a Markov kernel. Even in this setting, we can still demonstrate that the results obtained by \cite{Tie98} hold in the following lemma:

\begin{lemma}\label{lem:peskunpositive}
	Suppose that $P$ is a Markov kernel with stationary measure $\pi$, with $M_1$ and $M_2$ being the first and second MH kernel of $P$ respectively. We have the following:
	\begin{enumerate}[label={\upshape(\roman*)}, align=left, widest=iii, leftmargin=*]
		\item $M_2 \succeq P \succeq M_1$.
		\item $P - M_2$, $M_1 - P$ and $M_1 - M_2$ are positive semidefinite operators.
	\end{enumerate}
\end{lemma}

\begin{proof}
	\begin{enumerate}[label={\upshape(\roman*)}, align=left, widest=iii, leftmargin=*]
			\item For $x \in \mathcal{X}$ and $A \in \mathcal{F}$ with $x \notin A$,
			\begin{align*}
				M_2(x,A) = \int_{A} \max\{p(x,y),p^*(x,y)\} \mu(dy)
						 \geq P(x,A)
						 \geq \int_{A} \min\{p(x,y),p^*(x,y)\} \mu(dy)
						 = M_1(x,A).
			\end{align*}
			
			\item We modify the proof of Lemma $3$ in \cite{Tie98} to cater for the case where $M_2$ may not be a Markov kernel. Let $H(dx,dy) = \pi(dx) (\delta_x(dy) - P(x,dy) + M_2(x,dy))$. Lemma \ref{lem:M1M2prop} yields $H(A) \geq 0$ , $H(\mathcal{X} \times \mathcal{X}) = 1$, $H(\mathcal{X} \times B) = H(B \times \mathcal{X}) = \pi(B)$ for $A \in \mathcal{F} \otimes \mathcal{F}$ and $B \in \mathcal{F}$. The rest of the proof are the same as the proof of Lemma $3$ in \cite{Tie98}.
	\end{enumerate}
\end{proof}

\begin{corollary}\label{cor:specord}
	Suppose that $P$ is a Markov kernel with stationary measure $\pi$, with $M_1$ and $M_2$ being the first and second MH kernel of $P$ respectively. Using the notation defined in \eqref{eq:luspec}, we obtain:
	\begin{align}
		&\lambda(M_2) \leq \inf_{\substack{f \in L^2_0(\pi) \\ ||f||_{\pi}=1}} \langle Pf,f \rangle_{\pi} \leq \lambda(M_1), \\
		&\Lambda(M_2) \leq \sup_{\substack{f \in L^2_0(\pi) \\ ||f||_{\pi}=1}} \langle Pf,f \rangle_{\pi} \leq \Lambda(M_1).
	\end{align}
\end{corollary}

\begin{proof}
	Lemma \ref{lem:peskunpositive} leads to
	$$\langle M_2 f,f \rangle_{\pi} \leq \langle Pf,f \rangle_{\pi} \leq \langle M_1 f,f \rangle_{\pi}.$$
	Desired result follows by taking infimum or supremum over $\{f \in L^2_0(\pi) : ||f||_{\pi}=1\}$, \eqref{eq:luspec} and Lemma \ref{lem:M1M2prop}(i).
\end{proof}

Inspired by Corollary \ref{cor:specord} and \eqref{eq:specgap_Prev}, we will introduce a \textit{pseudo}-spectral gap (that we will call the MH-spectral gap) based on $\lambda(M_2)$ and $\Lambda(M_1)$ in Section \ref{sec:geomergodMHspecgap}. We will obtain a number of new bounds based on this gap.

\subsection{Weyl's inequality for additive reversiblization}\label{sec:weyl}

In this section, we introduce Weyl's inequality for the additive reversiblization for finite Markov chains, which allows us to give upper and lower bound on the eigenvalues of $(P+P^*)/2$, in terms of the eigenvalues of $M_1(P)$ and $M_2(P)$. Assume that $P$ is a stochastic matrix on a finite state space $\mathcal{X}$ with stationary distribution $\pi$, with eigenvalues-eigenvectors denoted by $(\lambda_j(P),\phi_j(P))_{j=1}^{|\mathcal{X}|}$. If $P$ is a self-adjoint matrix, we arrange its eigenvalues in non-increasing order by $\lambda_1(P) \geq \ldots \geq \lambda_n(P)$, where $n := |\mathcal{X}|$. We write $l^2(\pi)$ to be the Hilbert space of square-summable function with respect to $\pi$. 

\begin{theorem}[Weyl's inequality for additive reversiblization]\label{thm:weyl}
	Assume that $P$ is a $n \times n$ stochastic matrix with stationary distribution $\pi$, with $M_1$, $M_2$ to be the first and second MH kernel.
	\begin{enumerate}[label={\upshape(\roman*)}, align=left, widest=iii, leftmargin=*]
		\item For integers $i,j,k$ such that $1 \leq i,j,k \leq n$ and $i+1 = j+k$,
		$$\lambda_i(P+P^*) \leq \lambda_j(M_1) + \lambda_k(M_2).$$
		Equality holds if and only if there exists a vector $f$ with $||f||_{l^2(\pi)} = 1$ such that $M_1f = \lambda_j f, M_2f = \lambda_k f$ and $(P+P^*)f = \lambda_i f$.
		
		\item For integers $i,l,m$ such that $1 \leq i,l,m \leq n$ and $i+n = l+m$,
		$$\lambda_i(P+P^*) \geq \lambda_l(M_1) + \lambda_m(M_2).$$
		Equality holds if and only if there exists a vector $f$ with $||f||_{l^2(\pi)} = 1$ such that $M_1f = \lambda_l f, M_2f = \lambda_m f$ and $(P+P^*)f = \lambda_i f$.
	\end{enumerate}
\end{theorem}

\begin{proof}
	Thanks to Lemma \ref{lem:M1M2prop}\ref{state:1}, $P + P^* = M_1 + M_2$, where both $M_1$ and $M_2$ are self-adjoint matrices in $l^2(\pi)$. Desired results follow directly from Weyl's inequality, see e.g.~Theorem $4.3.1$ in \cite{HJ13}.
\end{proof}

Since $\gamma(P) = \gamma((P+P^*)/2)$, we can obtain bounds on the spectral gap of $P$ in terms of the eigenvalues of $M_1, M_2$.

\begin{corollary}\label{cor:weyl}
	With the assumptions of Theorem \ref{thm:weyl}, we have
	$$1 - \dfrac{1}{2}U \leq \gamma(P) \leq 1-\dfrac{1}{2} L,$$
	where $L := \max_{l+m = 2+n} \{\lambda_l(M_1) + \lambda_m(M_2)\}$ and $U := \min_{j+k=3} \{\lambda_j(M_1) + \lambda_k(M_2)\}$.
\end{corollary}

\begin{proof}
	We take $i = 2$ in Theorem \ref{thm:weyl} to obtain 
	$$\dfrac{1}{2} L \leq \lambda_2((P+P^*)/2) \leq \dfrac{1}{2} U.$$
	Desired result follows by using $\gamma(P) = \gamma((P+P^*)/2) = 1 - \lambda_2((P+P^*)/2)$.
\end{proof}

\subsection{Examples: asymmetric random walk and birth-death processes with vortices}

In this section, we first show that the bounds in Corollary \ref{cor:weyl} are sharp by studying the asymmetric simple random walk on $n$-cycle and on discrete torus. We then proceed to give spectral gap bounds for birth-death processes with vortices.

\begin{example}[Asymmetric simple random walk on the $n$-cycle]\label{ex:srw}
	We first recall the asymmetric simple random walk on the $n$-cycle. We take $\mathcal{X} = \{0,1,\ldots,n-1\}$ and the transition
	matrix to be $P(j,k) = p$ for $k = j + 1 \mod{n}$, $P(j,k) = q = 1 - p$ for $k = j - 1 \mod{n}$ and $0$ otherwise. Its stationary distribution is given by $\pi(i) = 1/n$ for all $i \in \mathcal{X}$, and its time-reversal has transition matrix given by $P^* = P^T$, the transpose of $P$. In the particular case when $p = q = 1/2$, we recover the symmetric random walk with eigenvalues $(\cos(2\pi j/n))_{j=0}^{n-1}$, which have been studied in \cite{LPW09, Fill91, DS91}.
	
	We denote $l := \min\{p,q\}$ and $r := \max\{p,q\}$. Then $M_1$ and $M_2$ are given by, for $j \in \mathcal{X}$,
	\begin{align*}
	M_1(j,k) &= l, \quad \mathrm{for} \, k = j \pm 1 \mod{n}, \quad
	M_1(j,j) = 1-2l, \\
	M_2(j,k) &= r, \quad \mathrm{for} \, k = j \pm 1 \mod{n}, \quad
	M_2(j,j) = 1-2r.
	\end{align*}
	Note that $M_2$ is not a Markov kernel unless $r = p = q = 1/2$. For $p \neq 1/2$, we can interpret $M_2$ as $M_2 = G + I$, where $G := M_2 - I$ is the Markov generator on $\mathcal{X}$. Using the notation of Section \ref{sec:weyl} and observe that the additive reversiblization is the simple symmetric random walk, the unordered eigenvalues of $(P+P^*)/2$, $M_1$ and $M_2$ (see Example $3.1, 3.2$ in \cite{Fill91}) are, for $i \in \{1,\ldots,n\}$,
	\begin{align*}
	\lambda_i((P+P^*)/2) &= \cos(2\pi (i-1)/n), \\
	\lambda_i(M_1) &= 1 - 2l (1-\cos(2\pi (i-1)/n)), \\
	\lambda_i(M_2) &= 1 - 2r (1-\cos(2\pi (i-1)/n)),
	\end{align*}
	so Corollary \ref{cor:weyl} now reads $L=2\cos(2\pi/n)$, $U = 2 - 2r (1-\cos(2\pi/n))$ and
	$$r(1 - \cos (2\pi/n)) = 1 - \dfrac{1}{2}U \leq \gamma(P) = 1 - \cos (2\pi/n) = 1 - \dfrac{1}{2}L,$$
	that is, the upper bound is exactly attained and the lower bound is sharp in $n$.
\end{example}

\begin{example}[Asymmetric simple random walk on discrete torus]\label{ex:torus}
	This example investigates the asymmetric simple random walk on discrete torus $\mathbb{Z}_n^d = (\mathbb{Z} \backslash n \mathbb{Z})^d$, in which we build a product chain via the asymmetric kernel on the $n$-cycle that we studied in Example \ref{ex:srw} and we also adapt the notations therein. That is, we choose one of the $d$ coordinates at random and it will move according to the kernel $P(j,k) = p$ for $k = j + 1 \mod{n}$, $P(j,k) = q = 1 - p$ for $k = j - 1 \mod{n}$ and $0$ otherwise. Denote the Markov kernel (resp.~first Metropolis kernel, second Metropolis kernel) on $\mathbb{Z}_n^d$ by $\widetilde{P}$ (resp.~$\widetilde{M_1}$, $\widetilde{M_2}$), then we have
	\begin{align*}
		\widetilde{P} &= \dfrac{1}{d} \sum_{i=1}^d \underbrace{I \otimes \cdots \otimes I}_{i-1} \otimes P \otimes \underbrace{I \otimes \cdots \otimes I}_{d-i}, \\
		\widetilde{M_1} &= \dfrac{1}{d} \sum_{i=1}^d \underbrace{I \otimes \cdots \otimes I}_{i-1} \otimes M_1 \otimes \underbrace{I \otimes \cdots \otimes I}_{d-i}, \\
		\widetilde{M_2} &= \dfrac{1}{d} \sum_{i=1}^d \underbrace{I \otimes \cdots \otimes I}_{i-1} \otimes M_2 \otimes \underbrace{I \otimes \cdots \otimes I}_{d-i}.		
	\end{align*}
	Note that the stationary distribution is the uniform distribution on $\mathbb{Z}_n^d$. The unordered eigenvalues of $(\widetilde{P}+\widetilde{P}^*)/2$, $\widetilde{M_1}$ and $\widetilde{M_2}$ are, for $i \in \{1,\ldots,n\}$,
	\begin{align*}
	\lambda_i((\widetilde{P}+\widetilde{P}^*)/2) &= (d-1)/d + \cos(2\pi (i-1)/n)/d, \\
	\lambda_i(\widetilde{M_1}) &= 1 - 2l(1-\cos(2\pi (i-1)/n))/d, \\
	\lambda_i(\widetilde{M_2}) &= 1 - 2r (1-\cos(2\pi (i-1)/n))/d.
	\end{align*}
	and Corollary \ref{cor:weyl} now reads $L=2-2/d+2\cos(2\pi/n)/d$, $U = 2 - 2r (1-\cos(2\pi/n))/d$ and
	$$r(1 - \cos (2\pi/n))/d = 1 - \dfrac{1}{2}U \leq \gamma(P) = \dfrac{1-\cos(2\pi/n)}{d} = 1 - \dfrac{1}{2}L,$$
	that is, the upper bound is exactly attained and the lower bound is sharp in $n$.
\end{example}

\begin{example}[Inserting vortices to birth-death processes]
	Giving two-sided precise spectral gap bounds for non-reversible Markov chains is well-known to be a difficult task. For the spectral gap estimates of birth-death processes, we refer interested readers to \cite{Chen96}. We aim at using this example to show how we can give such type of estimates by means of MH reversibilization. This example is inspired by \cite{SGS10, Bie16}, which offers an interesting way to artificially create non-reversible Markov chains from reversible ones via perturbation or inserting vortices.
	It is perhaps more suitable to work in the setting of continuous-time Markov chains. We write $G^{BD}$ to be the infinitesimal generator of a birth-death process with birth rate $b_i$ and death rate $d_i$ for $i \in \mathcal{X} = \mathbb{N}_0$ with stationary distribution $\pi(i)$. Next, we denote $V$ to be the $n$-dimensional cyclic vortices given by $V(i,i) = -1/\pi(i)$ and $V(i,j) = 1/\pi(i)$ for $j = (i+1) \mod{n}$ for $i \in \{0,\ldots,n\}$. By Corollary $1$ in \cite{SGS10}, $G := G^{BD} + V$ is the generator of a non-reversible Markov chain on $\mathcal{X}$ with stationary distribution $\pi$.
	
	To analyze the left spectral gap $\gamma(G)$ of $-G$, the construction of $M_1$ and $M_2$ applies essentially in verbatim to $G$ as in Section \ref{sec:MHr}. More precisely, we define $\gamma(G)$ to be the smallest distance between the spectrum of $-G$ to $0$, i.e.
	$$\gamma(G) := - \sup_{\substack{f \in L^2_0(\pi) \\ ||f||_{\pi}=1}} \langle Gf,f \rangle_{\pi} = \gamma((G+G^*)/2),$$
	where we used $\langle Gf,f \rangle_{\pi} = \langle ((G+G^*)/2)f,f \rangle_{\pi}$ in the equality above. Note that the Peskun ordering of generator gives 
	$$\langle M_2(G)f,f \rangle_{\pi} \leq \langle Gf,f \rangle_{\pi} \leq \langle M_1(G)f,f \rangle_{\pi}.$$
	Specializing the above into our examples of birth-death processes with vortices, we take $G^* = G^{BD} + V^*$, $M_1(G) = G^{BD}$ and $M_2(G) = G^{BD} + V + V^*$. As a result, it follows that
	$$\gamma(G^{BD}) \leq \gamma(G) \leq \gamma(G^{BD}) + \gamma(V+V^*) \leq \gamma(G^{BD}) + \dfrac{2}{\min_{i \in \{0,\ldots,n\}} \pi(i)} (1 -  \cos(2\pi/n)), $$
	where we further upper bound the left spectral gap of $V+V^*$ by the symmetric random walk on $n$-cycle with birth and death rate $1/\min_{i \in \{0,\ldots,n\}} \pi(i)$. We can then specialize into various well-known examples of birth-death processes, in which we summarize the results below:
	\begin{table}[H]
		\centering
		\begin{tabular}{cccc}
			\toprule
			Process  & spectral gap bounds \\
			\midrule
			Ehrenfest with vortices & $1 \leq \gamma(G) \leq 1 + 2 (1-\cos(2\pi/n)) \max\{p^{-n},(1-p)^{-n}\}$ \\
			$M/M/1$ with vortices & $(\sqrt{\mu} - \sqrt{\lambda})^2 \leq \gamma(G) \leq (\sqrt{\mu} - \sqrt{\lambda})^2 + 2(1-\lambda/\mu)^{-1}(\mu/\lambda)^{n-1}(1-\cos(2\pi/n))$  \\
			$M/M/\infty$ with vortices & $1 \leq \gamma(G) \leq 1 + 2 (1-\cos(2\pi/n)) e^{\lambda} \max_{i \in \{0,\ldots,n\}} i! \lambda^{-i}$  \\
			GWI with vortices & $1 - \lambda \leq \gamma(G) \leq 1 - \lambda + 2 (1-\cos(2\pi/n)) \max_{i \in \{0,\ldots,n\}} \dfrac{\Gamma(r)i!}{\Gamma(r+i)} (1-\lambda)^{-r} \lambda^{-i}$   \\
			\bottomrule
		\end{tabular}
		\captionsetup{font=footnotesize}
		\caption{Spectral gap bounds for various birth-death processes with $n$-dimensional cyclic vortices}
		\label{tab:addlabel}%
	\end{table}
	For the Ehrenfest model with cyclic vortices, it is constructed from a birth-death process with $b_i = p(n-i)$, $d_i = (1-p)i$ with $0 < p < 1$ on $\mathcal{X} = \{0,\ldots,n\}$ and $\pi$ being the binomial distribution with parameters $n$ and $p$. For $M/M/1$ with vortices, it is constructed from a birth-death process with $b_i = \lambda$, $d_i = \mu$ with $\mu > \lambda$ and $\pi(i) = (1-\lambda/\mu)(\lambda/\mu)^{i-1}$. For $M/M/\infty$ with vortices, it is constructed from a birth-death process with $b_i = \lambda$, $d_i = i$ and $\pi$ being the Poisson distribution with mean $\lambda$. For the Galton-Watson process with immigration (GWI) and vortices, we have $b_i = \lambda(r+i)$, $d_i = i$ and $\pi$ being the negative binomial distribution with parameters $\lambda$ and $r$.
	
	In the literature, the reciprocal of $\gamma(G)$ is commonly known as the relaxation time, which serves as a lower bound in the total variation mixing time of the chain, see for instance \cite[Theorem $1$]{J13}. This means that the reciprocal of the upper bound of the spectral gap in the table above can be used to give lower bound on the total variation mixing time. In Section \ref{sec:geomergodMHspecgap}, we formally introduce various notions of ergodicity of Markov chain as well as total variation mixing time. In this spirit, we would like to mention the work of \cite[Section $4$]{Fill91} and \cite{Chen96}. In the former, the author studied upper and lower bounds of the spectral gaps of exclusion processes, while in the latter the author investigated two-sided spectral gap bounds of classical birth-death processes.
\end{example}
%
%
%
%
\section{Pseudospectral expansion}\label{sec:pseexp}
As a consequence of Lemma \ref{lem:M1M2prop}\ref{state:2}, $M_1$ and $M_2$ are self-adjoint operators on $L^2(\pi)$, which help us to obtain a \textit{pseudospectral} expansion of $P$ in terms of the spectral measures of $M_1$ and $M_2$.

\begin{theorem}\label{thm:pseexp}
	Denote $P$ to be a Markov kernel with stationary distribution $\pi$, and $M_i$ to be the MH kernel (defined in def. \ref{def:M1} and def. \ref{def:M2}) with spectral measure $\mathcal{E}_i$ for $i= 1,2$. For $x \in \mathcal{X}$, $B \in \mathcal{F}$ with $x \notin B$, we have
	\begin{align}
		P(x,B) &= \int_{\sigma(M_1)} \lambda \, \delta_x \mathcal{E}_1(d \lambda) (B \cap A_x^c) +  \int_{\sigma(M_2)} \lambda \, \delta_x \mathcal{E}_2(d \lambda) (B \cap A_x), \label{eq:psexB}\\
		P(x,\{x\}) &= \dfrac{1}{2} \left( \int_{\sigma(M_1)} \lambda \, \delta_x \mathcal{E}_1(d \lambda) (\{x\} ) + \int_{\sigma(M_2)} \lambda \, \delta_x \mathcal{E}_2(d \lambda) (\{x\}) \right) \label{eq:psexx}, \\
		P^*(x,B) &= \int_{\sigma(M_1)} \lambda \, \delta_x \mathcal{E}_1(d \lambda) (B \cap A_x) + \int_{\sigma(M_2)} \lambda \, \delta_x \mathcal{E}_2(d \lambda) (B \cap A_x^c), \label{eq:p*sexB}	
	\end{align}
	where we recall that $A_x := \{y \in E : \alpha_1(x,y) < 1\}$.
\end{theorem}

\begin{proof}
We first show \eqref{eq:psexB}. By Lemma \ref{lem:M1M2prop} and \eqref{eq:specthm}, for $i=1,2$,
\begin{align}\label{eq:Mispecthm}
	M_i &= \int_{\sigma(M_i)} \lambda \, \mathcal{E}_i(d \lambda).
\end{align}
Therefore, we can deduce that
\begin{align*}
	P(x,B) &= P(x,B \cap A_x^c) + P(x,B \cap A_x) \\
		   &= \int_{B \cap A_x^c} P(x,dy) + \int_{B \cap A_x} P(x,dy) \\
		   &= \int_{B \cap A_x^c} M_1(x,dy) + \int_{B \cap A_x} M_2(x,dy) \quad (\1_B(x) = 0) \\
		   &= \delta_x M_1(B \cap A_x^c) + \delta_x M_2(B \cap A_x) \\
		   &= \int_{\sigma(M_1)} \lambda \, \delta_x \mathcal{E}_1(d \lambda) (B \cap A_x^c) +  \int_{\sigma(M_2)} \lambda \, \delta_x \mathcal{E}_2(d \lambda) (B \cap A_x) . \quad \text{(By \eqref{eq:Mispecthm})}
\end{align*}
Next, in view of Lemma \ref{lem:M1M2prop}, we have
\begin{align*}
	P(x,\{x\}) &= \dfrac{1}{2} \left(M_1(x,\{x\}) + M_2(x,\{x\}) \right)\\
	&= \dfrac{1}{2} \left( \int_{\sigma(M_1)} \lambda \, \delta_x \mathcal{E}_1(d \lambda) (\{x\} ) + \int_{\sigma(M_2)} \lambda \, \delta_x \mathcal{E}_2(d \lambda) (\{x\}) \right),
\end{align*}
which gives \eqref{eq:psexx}. Finally, to show \eqref{eq:p*sexB}, we follow a very similar proof of \eqref{eq:psexB} that leads to
\begin{align*}
	P^*(x,B) &= P^*(x,B \cap A_x) + P^*(x,B \cap A_x^c) \\
	&= \delta_x M_1(B \cap A_x) + \delta_x M_2(B \cap A_x^c) \\
	&= \int_{\sigma(M_1)} \lambda \, \delta_x \mathcal{E}_1(d \lambda) (B \cap A_x) + \int_{\sigma(M_2)} \lambda \, \delta_x \mathcal{E}_2(d \lambda) (B \cap A_x^c) .
\end{align*}
\end{proof}

\begin{rk}
When $P$ is reversible, Lemma \ref{lem:M1M2prop} yields $P = M_1 = M_2$, so \eqref{eq:psexB} and \eqref{eq:psexx} reduces to
		\begin{align*}
		P(x,B) &= \int_{\sigma(P)} \lambda \, \delta_x \mathcal{E}_1(d \lambda) (B), \\
		P(x,\{x\}) &= \int_{\sigma(P)} \lambda \, \delta_x \mathcal{E}_1(d \lambda) (\{x\}) ,
		\end{align*}
		which are expected expressions (since we can invoke the spectral theorem directly on $P$).
\end{rk}

\begin{rk}
 An alternative expression for $P(x,\{x\})$ is the following: Using \eqref{eq:psexB} (with $B$ replaced by $\mathcal{X} \backslash \{x\}$), we observe that
		\begin{align*}
		P(x,\{x\}) &= 1 - P(x,\mathcal{X} \backslash \{x\}) \\
		&=1 - \int_{\sigma(M_1)} \lambda \, \delta_x \mathcal{E}_1(d \lambda) (\mathcal{X} \backslash \{x\} \cap A_x^c) -  \int_{\sigma(M_2)} \lambda \, \delta_x \mathcal{E}_2(d \lambda) (\mathcal{X} \backslash \{x\} \cap A_x) \\
		&= \int_{\sigma(M_1)} \lambda \, \delta_x \mathcal{E}_1(d \lambda) (\{x\} \cup A_x) -  \int_{\sigma(M_2)} \lambda \, \delta_x \mathcal{E}_2(d \lambda) (A_x).
		\end{align*}

\end{rk}
	
	To compute the pseudospectral expansion of the $n$-step Markov kernel $P^n$, we can make use of the Chapman-Kolmogorov equation together with \eqref{eq:psexB} and \eqref{eq:psexx}. Equivalently, we can replace $P$ by $P^n$ in Theorem \ref{thm:pseexp}, which leads to:
		
	\begin{corollary}
		Denote $P$ to be a Markov kernel with stationary measure $\pi$, so that $P^n$ is the n-step Markov kernel for $n \in \mathbb{N}$. Let $M_i(P^n)$ to be the MH kernel (defined in def. \ref{def:M1} and def. \ref{def:M2}) with spectral measure $\mathcal{E}_i(P^n)$ for $i= 1,2$. For $x \in \mathcal{X}$, $B \in \mathcal{F}$ with $x \notin B$, we have
				\begin{align}
					P^n(x,B) &= \int_{\sigma(M_1(P^n))} \lambda \, \delta_x \mathcal{E}_1(P^n)(d \lambda) (B \cap A_x^c) +  \int_{\sigma(M_2(P^n))} \lambda \, \delta_x \mathcal{E}_2(P^n)(d \lambda) (B \cap A_x), \label{eq:pnsexB}\\
					P^n(x,\{x\}) &= \dfrac{1}{2} \left( \int_{\sigma(M_1(P^n))} \lambda \, \delta_x \mathcal{E}_1(P^n)(d \lambda) (\{x\} ) + \int_{\sigma(M_2(P^n))} \lambda \, \delta_x \mathcal{E}_2(P^n)(d \lambda) (\{x\}) \right) \label{eq:pnsexx}, \\
					P^{*n}(x,B) &= \int_{\sigma(M_1(P^n))} \lambda \, \delta_x \mathcal{E}_1(P^n)(d \lambda) (B \cap A_x) + \int_{\sigma(M_2(P^n))} \lambda \, \delta_x \mathcal{E}_2(P^n)(d \lambda) (B \cap A_x^c), \label{eq:p*nsexB}	
				\end{align}
				where $A_x := \{y : \alpha_1(P^n)(x,y) < 1\} $.
		\end{corollary}

	Next, we specialize into the case of finite Markov chains, as more explicit results can be obtained.
	
	\begin{corollary}
		Suppose that $P$ is a Markov kernel on a finite state space $\mathcal{X}$ with stationary distribution $\pi$. Let $M_i(P^n)$ be the MH kernel with eigenvalues-eigenvectors denoted by $(\lambda^{(i)}_j,\phi^{(i)}_j)_{j=1}^{|\mathcal{X}|}$ for $i= 1,2$ (note that the dependence of $(\lambda^{(i)}_j,\phi^{(i)}_j)$ on $P^n$ is suppressed). For $x \in \mathcal{X}$ and $f \in l^2(\pi)$, we have
		\begin{align}
		 P^n(x,y) &= \begin{cases} \sum_{j = 1}^{|\mathcal{X}|} \lambda^{(2)}_j \phi^{(2)}_j(x) \phi^{(2)}_j(y) \pi(y) &\mbox{if } y \in A_x, \\
			\sum_{j = 1}^{|\mathcal{X}|} \lambda^{(1)}_j \phi^{(1)}_j(x) \phi^{(1)}_j(y)\pi(y) & \mbox{if } y \in A_x^c \backslash \{x\} , \\
			\dfrac{1}{2}(\sum_{j = 1}^{|\mathcal{X}|} \lambda^{(1)}_j \phi^{(1)}_j(x) \phi^{(1)}_j(x) \pi(x) + \sum_{j = 1}^{|\mathcal{X}|} \lambda^{(2)}_j \phi^{(2)}_j(x) \phi^{(2)}_j(x) \pi(x)) & \mbox{if } y = x. \end{cases} \label{eq:pnsexy}\\
		P^nf(x) &= \sum_{j = 1}^{|\mathcal{X}|} \lambda^{(1)}_j \phi^{(1)}_j(x) \langle f \1_{A_x^c \backslash \{x\}}, \phi^{(1)}_j \rangle_{\pi} + \sum_{j = 1}^{|\mathcal{X}|} \lambda^{(2)}_j \phi^{(2)}_j(x) \langle f \1_{A_x}, \phi^{(2)}_j \rangle_{\pi} + P^n(x,x) f(x) \label{eq:pnfsex},
		\end{align}
		where $A_x := \{y : \alpha_1(P^n)(x,y) < 1\} $.
	\end{corollary}
	
	\begin{proof}
		The proof of \eqref{eq:pnsexy} follows from \eqref{eq:pnsexB} and \eqref{eq:pnsexx}. To see \eqref{eq:pnfsex}, we decompose $P^nf(x)$ into
		\begin{align}
			P^nf(x) &= P^n f \1_{A_x^c \backslash \{x\}}(x) + P^n f \1_{A_x}(x) + P^n f \1_{\{x\}}(x) \notag \\
				    &= M_1(P^n) f \1_{A_x^c \backslash \{x\}}(x) + M_2(P^n) f \1_{A_x}(x) + P^n f \1_{\{x\}}(x) \label{eq:Pnf}.
		\end{align}
		Now, since $(\phi^{(i)}_j)_{j=1}^{|\mathcal{X}|}$ is a basis on $l^2(\pi)$ for $i = 1,2$, we have
		\begin{align}
			M_1(P^n) f \1_{A_x^c \backslash \{x\}}(x) &= \sum_{j = 1}^{|\mathcal{X}|} \lambda^{(1)}_j \phi^{(1)}_j(x) \langle f \1_{A_x^c \backslash \{x\}}, \phi^{(1)}_j \rangle_{\pi}, \label{eq:M1f} \\
			M_2(P^n) f \1_{A_x}(x) &= \sum_{j = 1}^{|\mathcal{X}|} \lambda^{(2)}_j \phi^{(2)}_j(x) \langle f \1_{A_x}, \phi^2_j \rangle_{\pi} . \label{eq:M2f}
		\end{align}
		Desired result follows by collecting \eqref{eq:Pnf}, \eqref{eq:M1f} and \eqref{eq:M2f}.
	\end{proof}
	
\section{Geometric ergodicity, mixing time and MH-spectral gap}\label{sec:geomergodMHspecgap}

We will measure the speed of convergence to stationarity by the total variation distance, which is defined to be:

\begin{definition}[Total variation distance]\label{def:tv}
	The total variation distance between two signed measures $\mu$ and $\nu$ is given by
	\begin{align*}
		||\mu - \nu||_{TV} := \sup_{A} |\mu(A) - \nu(A)| = \dfrac{1}{2} \sup_{|f| \leq 1} |\mu(f) - \nu(f)|.
	\end{align*}
	where $|f| := \sup_{x \in \mathcal{X}} |f(x)|$.	
\end{definition}

We refer the readers to \cite{LPW09}, \cite{MT09} and \cite{RR04} for further properties of the total variation distance. In our main results in Section \ref{subsec:main} below, we are primarily interested in the case when $\mu$ is either $P(x,\cdot)$, $P^*(x,\cdot)$, $M_1(x,\cdot)$ or $M_2(x,\cdot)$ for $x \in \mathcal{X}$, while $\nu$ is taken to be the stationary measure $\pi$. We can now define various notions of ergodicity of a general kernel $K$.

\begin{definition}[Geometric ergodic, $\pi$-a.e. geometrically ergodic, uniformly ergodic, mixing time]\label{def:geomergod}
	Suppose that $K$ is a general kernel with stationary measure $\pi$, that is, $\pi K = \pi$. $K$ is geometrically ergodic if for each probability measure $\mu$, there exists $C_\mu< \infty$ and $\rho \in [0,1)$ such that
	\begin{align}\label{eq:geomergod}
		||\mu K^n - \pi ||_{TV} \leq C_{\mu} \rho^n, \quad n \in \mathbb{N}.
	\end{align}
	If \eqref{eq:geomergod} holds with $\mu = \delta_x$ for $\pi$-a.e. $x$, then $K$ is called $\pi$-a.e. geometrically ergodic. $K$ is uniformly ergodic if there exists $C < \infty$ and $\rho \in [0,1)$ such that
	\begin{align}\label{eq:unifergod}
	d(n) := \sup_{x \in \mathcal{X}}||K^n(x,\cdot) - \pi ||_{TV} \leq C \rho^n, \quad n \in \mathbb{N}.
	\end{align}
	The mixing time $t_{mix}(\epsilon)$ is defined to be
	$$t_{mix}(\epsilon) := \min \{ n : d(n) \leq \epsilon\}.$$
\end{definition}

Similar to the remarks after Definition \ref{def:tv} above, in subsequent sections we consider the case when the kernel $K$ is taken to be either $P, P^*, M_1$ or $M_2$. In particular, it is because of $M_2$ that we introduce these general notions of ergodicity for non-Markovian kernels. For Markov kernels that are reversible w.r.t. $\pi$, we have the following characterization of geometric ergodicity in terms of the $L^2$-spectral gap.

\begin{theorem}\label{thm:RR97}
	Suppose that $P$ is reversible with respect to $\pi$. The following statements are equivalent:
	\begin{enumerate}[label={\upshape(\roman*)}, align=left, widest=iii, leftmargin=*]
		\item $P$ is geometrically ergodic.
		
		\item $P$ admits a $L^2$-spectral gap, i.e. $\gamma^* = 1 - \beta = 1 - \max\{|\lambda|,\Lambda\} > 0.$
	\end{enumerate}
\end{theorem}

The proof can be found in \citet{RR97}.

\subsection{Main results}\label{subsec:main}

Following from the result in Corollary \ref{cor:specord} and \eqref{eq:specgap_Prev}, we can define a \textit{pseudo}-spectral gap by taking $1 - \max\{|\lambda(M_2)|,\Lambda(M_1)\}$. However, this gap may not be informative as $|\lambda(M_2)|$ maybe greater than or equal to $1$ since $M_2$ is not a Markov kernel in general. To define a meaningful gap, we should consider $M_2$ with $|\lambda(M_2)| < 1$. This leads us to the following definition:

\begin{definition}[MH-spectral gap]\label{def:MHspecgap}
	Suppose that $P$ is a Markov kernel with stationary measure $\pi$. Let
	\begin{align*}
		\mathcal{C} &:= \{ n \in \mathbb{N} : |\lambda(M_2(P^n))| < 1, \Lambda(M_1(P^n)) < 1 \}, \quad \mathcal{C}^c := \mathbb{N} \backslash \mathcal{C},\\
		\beta^{MH}  &:= \begin{cases}
		1, &\text{if $\mathcal{C} = \emptyset$,}\\
		\sup_{n \in \mathcal{C}}\{ |\lambda(M_2(P^n))|^{1/n}, \Lambda(M_1(P^n))^{1/n} \}, &\text{otherwise.}
		\end{cases}
	\end{align*}
	The MH-spectral gap $\gamma^{MH} = \gamma^{MH}(P)$ is given by
	$$\gamma^{MH} := 1 - \beta^{MH}.$$
\end{definition}

In this definition, we insert the idea of ``burn-in" in MCMC to define a MH-spectral gap. Precisely, we discard the spectral gaps in $\mathcal{C}^c$, and only consider the gaps in $\mathcal{C}$.

Note that for reversible $P$, the $L^2$-spectral gap is equal to the MH-spectral gap. If $P$ is geometrically ergodic, Lemma \ref{lem:M1M2prop}\ref{state:3} and Theorem \ref{thm:RR97} lead to $\mathcal{C} = \mathbb{N}$ and
\begin{align*}
	\gamma^{MH} &= 1 - \beta^{MH} \\
				&= 1 - \sup_{n \in \mathbb{N}}\{ |\lambda(P^n)|^{1/n}, \Lambda(P^n)^{1/n} \} \\
			    &= 1 - \sup_{n \in \mathbb{N}}\{ |\lambda(P)|, \Lambda(P) \} \\
			    &= 1 - \max \{ |\lambda(P)|, \Lambda(P) \} = 1 - \beta = \gamma^*.
\end{align*}
If $P$ is not geometrically ergodic, then $\beta^{MH} = \beta = 1$, so $\gamma^{MH} = \gamma^* = 0$.

As a first result, by means of Weyl's inequality, we can show that $M_2 = M_2(P)$ is a contraction whenever $P$ is a finite-state lazy and ergodic Markov kernel. Recall that a finite-state Markov chain is said to be lazy if $p(x,x) \geq 1/2$, and ergodic if $P$ is irreducible and aperiodic.

\begin{theorem}
	If $P$ is a finite-state lazy and ergodic Markov kernel, then $|\lambda(M_2(P))| \leq 1$.
\end{theorem}

\begin{proof}
	By Weyl's inequality introduced in Theorem \ref{thm:weyl}, we have 
	$$\lambda_n(P+P^*) \leq \lambda_1(M_1) + \lambda_n(M_2) = 1 + \lambda_n(M_2).$$
	Note that laziness of $P$ implies the laziness of $(P+P^*)/2$, which implies $0 \leq \lambda_n(P+P^*) \leq 1 + \lambda_n(M_2)$. On the other hand, using Corollary \ref{cor:specord} and the fact that $M_1$ is a Markov kernel, we have $\lambda_n(M_2) \leq \lambda_n(M_1) \leq 1$.
\end{proof}

Next, we present the main results in this section. Theorem \ref{thm:mainresultgeomergod} shows that a MH-spectral gap leads to geometric ergodicity.

\begin{theorem}\label{thm:mainresultgeomergod}
	Suppose that $P$ is the Markov kernel of a $\phi$-irreducible and aperiodic Markov chain with stationary measure $\pi$ on a countably generated state space $\mathcal{X}$. If $|\mathcal{C}^c| < \infty$ and $P$ admits a MH-spectral gap, i.e. $\gamma^{MH} = 1 - \beta^{MH} > 0$, then $P$ and $P^*$ are geometrically ergodic (and $\pi$-a.e. geometrically ergodic).
\end{theorem}

Next, we demonstrate a partial converse to Theorem \ref{thm:mainresultgeomergod}.

\begin{theorem}\label{thm:mainresultgeomergod2}[Partial converse of Theorem \ref{thm:mainresultgeomergod}]
	Suppose that $P$ is a Markov kernel with stationary measure $\pi$. If $P$ and $P^*$ are uniformly ergodic, then there exists a $N \in \mathbb{N}$ such that for all $n \geq N$, $M_i(P^n)$ are uniformly ergodic for $i=1,2$.
\end{theorem}

Recall that in the reversible case Theorem \ref{thm:RR97} gives the existence of $L^2$-spectral gap is equivalent to geometric ergodicity. While we hope for a result that characterizes geometric ergodicity in the non-reversible case by means of the MH-spectral gap, we only manage to show that under a stronger assumption of uniform ergodicity of both $P$ and $P^*$, $M_i(P^n)$ is uniformly ergodic for sufficiently large $n$. This implies the existence of an $L^2$-spectral gap of $M_i(P^n)$ for sufficiently large $n$, yet it is unclear whether $\beta^{MH}$ is less than $1$ (since we are taking supremum in the definition of $\beta^{MH}$).

Next, we present a result that gives a mixing time upper bound in terms of the MH-spectral gap.

\begin{corollary}\label{cor:mainresultgeomergod}
	For a finite Markov chain with Markov kernel $P$ that is irreducible, if $|\mathcal{C}^c| < \infty$ and $P$ admits a MH-spectral gap, then
	$$t_{mix}(\epsilon) \leq t^* + \dfrac{\log \left( \frac{1}{ \epsilon \pi_{min}}\right)}{\gamma^{MH}},$$
	where $t^* := \max_{n \in \mathcal{C}^c} n$ and $\pi_{min} := \min_x \pi(x)$.
\end{corollary}

Corollary \ref{cor:mainresultgeomergod} can be compared with the result in the reversible case (Theorem $12.3$ in \cite{LPW09}), which shows that
$$t_{mix}(\epsilon) \leq \dfrac{\log \left( \frac{1}{ \epsilon \pi_{min}}\right)}{\gamma^*}.$$
This also follows from Corollary \ref{cor:mainresultgeomergod} since $|\mathcal{C}^c| = 0$ and $\gamma^{MH} = \gamma^*$ in the reversible case.

\subsection{Proofs of Theorem \ref{thm:mainresultgeomergod}, Theorem \ref{thm:mainresultgeomergod2} and Corollary \ref{cor:mainresultgeomergod}}
First, we start with the following result that allows us to control the total variation distance of $P^n$ and $P^{*n}$ to $\pi$ by means of that of $M_1(P^n)$ and $M_2(P^n)$ and vice versa. The bounds are by no means tight, yet they will serve their purpose to demonstrate geometric ergodicity in the proof of Theorem \ref{thm:mainresultgeomergod} and Theorem \ref{thm:mainresultgeomergod2}.

\begin{lemma}\label{lem:tvPtvM1M2}
	Suppose that $P$ is a Markov kernel with stationary measure $\pi$, and $M_i$ to be the MH kernel for $i = 1,2$. For $x \in \mathcal{X}$, $n \in \mathbb{N}$,
	\allowdisplaybreaks
	\begin{align*}
		||P^n(x,\cdot) - \pi ||_{TV} &\leq \dfrac{3}{2} ||M_1(P^n)(x,\cdot) - \pi ||_{TV} + \dfrac{3}{2} ||M_2(P^n)(x,\cdot) - \pi ||_{TV}, \\
		||P^{*n}(x,\cdot) - \pi ||_{TV} &\leq \dfrac{3}{2} ||M_1(P^n)(x,\cdot) - \pi ||_{TV} + \dfrac{3}{2} ||M_2(P^n)(x,\cdot) - \pi ||_{TV}, \\		
		||M_1(P^n)(x,\cdot) - \pi ||_{TV} &\leq 2 ||P^n(x,\cdot) - \pi ||_{TV} + 2 ||P^{*n}(x,\cdot) - \pi ||_{TV}, \\
		||M_2(P^n)(x,\cdot) - \pi ||_{TV} &\leq 3 ||P^n(x,\cdot) - \pi ||_{TV} + 3 ||P^{*n}(x,\cdot) - \pi ||_{TV}. \\
	\end{align*}
\end{lemma}

\begin{proof} We use the same idea as in the proof of Theorem \ref{thm:pseexp}. For any $x \in E$ let $A_{x,n} := \{y : \alpha_1(P^n)(x,y) < 1\}$. We have, for all $n \in \mathbb{N}$,
	\allowdisplaybreaks
	\begin{align*}
		||P^n(x,\cdot) - \pi ||_{TV} &= \sup_{A} |P^n(x,A) - \pi(A)| \\
		&\leq \sup_{A} |P^n(x,A \cap A_{x,n}) - \pi(A \cap A_{x,n})| + \sup_{A} |P^n(x,A \cap A_{x,n}^c \backslash \{x\}) - \pi(A \cap A_{x,n}^c \backslash \{x\})| \\
		&\quad + |P^n(x,\{x\}) - \pi(\{x\})| \\
		&\leq ||M_2(P^n)(x,\cdot) - \pi ||_{TV} + ||M_1(P^n)(x,\cdot) - \pi ||_{TV} + \dfrac{1}{2} |M_1(P^n)(x,\{x\}) - \pi(\{x\})| \\
		&\quad + \dfrac{1}{2}|M_2(P^n)(x,\{x\}) - \pi(\{x\})| \\
		&\leq \dfrac{3}{2} ||M_1(P^n)(x,\cdot) - \pi ||_{TV} + \dfrac{3}{2} ||M_2(P^n)(x,\cdot) - \pi ||_{TV}.
	\end{align*}
	To show the inequality for $||P^{*n}(x,\cdot) - \pi ||_{TV}$, we replace $P$ by $P^*$ above and observe that $M_i(P^n) = M_i(P^{*n})$ for $i=1,2$ by Lemma \ref{lem:M1M2prop}\ref{state:4}. Next, we observe that
	\begin{align*}
	||M_1(P^n)(x,\cdot) - \pi ||_{TV} &\leq \sup_{|f| \leq 1} |M_1(P^n)(f \1_{A_{x,n}})(x) - \pi(f \1_{A_{x,n}})| \\
	&\quad + \sup_{|f| \leq 1} |M_1(P^n)(f \1_{A_{x,n}^c \backslash \{x\}})(x) - \pi(f \1_{A_{x,n}^c \backslash \{x\}})| \\
	&\quad + \sup_{|f| \leq 1} |M_1(P^n)(x,\{x\}) - \pi(\{x\})| |f(x)| \\
	&\leq ||P^{*n}(x,\cdot) - \pi ||_{TV} + ||P^n(x,\cdot) - \pi ||_{TV} + \sup_{|f| \leq 1}|P^n(x,A_{x,n}) - \pi(A_{x,n})||f(x)| \\
	&\quad + \sup_{|f| \leq 1}|P^{*n}(x,A_{x,n}^c \backslash\{x\}) - \pi(A_{x,n}^c \backslash\{x\})||f(x)| \\
	&= ||P^{*n}(x,\cdot) - \pi ||_{TV} + ||P^n(x,\cdot) - \pi ||_{TV} + \sup_{|f| \leq 1}|P^n(f(x)\1_{A_{x,n}})(x) - \pi(f(x)\1_{A_{x,n}})| \\
	&\quad + \sup_{|f| \leq 1}|P^{*n}(f(x)\1_{A_{x,n}^c \backslash\{x\}})(x) - \pi(f(x)\1_{A_{x,n}^c \backslash\{x\}})| \\
	&\leq 2 ||P^n(x,\cdot) - \pi ||_{TV} + 2 ||P^{*n}(x,\cdot) - \pi ||_{TV}.
	\end{align*}
	Finally, using the inequality above together with Lemma \ref{lem:M1M2prop}\ref{state:1} and the triangle inequality yields
	\begin{align*}
	||M_2(P^n)(x,\cdot) - \pi ||_{TV} &\leq ||P^n(x,\cdot) - \pi ||_{TV}+ ||P^{*n}(x,\cdot) - \pi ||_{TV} + ||M_1(P^n)(x,\cdot) - \pi ||_{TV} \\
	&\leq 3 ||P^n(x,\cdot) - \pi ||_{TV} + 3 ||P^{*n}(x,\cdot) - \pi ||_{TV}.
	\end{align*}	
\end{proof}

\begin{proof}[Proof of Theorem \ref{thm:mainresultgeomergod}]
	Fix $n \in \mathcal{C}$. Since $\beta^{MH} < 1$, both $M_1(P^n)$ and $M_2(P^n)$ admit $L^2$-spectral gap, that is, $\gamma(M_i(P^n)) > 0$ for $i= 1,2$. Theorem $2.1$ $(i) \rightarrow (iii)$ in \cite{RR97} gives that $M_1(P^n)$ and $M_2(P^n)$ are geometrically ergodic (even though $M_2(P^n)$ may not be a Markov kernel, the proof there will work through as long as $M_2(P^n)$ admits a $L^2$-spectral gap, since it follows from the Cauchy Schwartz inequality that the $L^1(\pi)$ norm is less than or equal to $L^2(\pi)$ norm). By Lemma \ref{lem:tvPtvM1M2}, we have
	\begin{align*}
		||P^n(x,\cdot) - \pi ||_{TV} &\leq \dfrac{3}{2} ||M_1(P^n)(x,\cdot) - \pi ||_{TV} + \dfrac{3}{2} ||M_2(P^n)(x,\cdot) - \pi ||_{TV} \\
		&\leq \dfrac{3}{2} C^1_x \beta(M_1(P^n)) + \dfrac{3}{2} C^2_x \beta(M_2(P^n)) \\
		&\leq \dfrac{3}{2} (C^1_x + C^2_x) (\beta^{MH})^n =  \dfrac{3}{2} (C^1_x + C^2_x) (1 - \gamma^{MH})^n,
	\end{align*}
	where $C^i_x$ are the constants of geometric ergodicity for $M_i(P^n)$ for $i=1,2$ as in Definition \ref{def:geomergod}, and the third inequality follows from Corollary \ref{cor:specord}. For $n \in \mathcal{C}^c$, we can bound it by a similar way. Precisely, let $\beta^{max} = \max_{n \in \mathcal{C}^c}\{ \beta(M_2(P^n)) \} = \max_{n \in \mathcal{C}^c}\{ |\lambda(M_2(P^n))| \}.$ Using again Lemma \ref{lem:tvPtvM1M2} leads to
	\begin{align*}
		||P^n(x,\cdot) - \pi ||_{TV} &\leq \dfrac{3}{2} ||M_1(P^n)(x,\cdot) - \pi ||_{TV} + \dfrac{3}{2} ||M_2(P^n)(x,\cdot) - \pi ||_{TV} \\
		&\leq \dfrac{3}{2} (C^1_x + C^2_x) \beta^{max} \\
		&\leq \dfrac{3}{2} (C^1_x + C^2_x) \dfrac{\beta^{max}}{(\beta^{MH})^{|\mathcal{C}^c|}}(\beta^{MH})^n .
	\end{align*}
	We have shown that $P$ is $\pi$-a.e. geometrically ergodic, and we can extend it to geometric ergodicity by adapting the argument in the last paragraph of page $9$ in \cite{RR97}  i.e. the direction from Proposition $1$ to Theorem $2$. (This is the place where we use the assumption of $\phi$-irreducibility and aperiodicity on a countably generated state space.) The proof of geometric ergodicity of $P^*$ is the same as above (by replacing $P$ by $P^*$) and is omitted.
%
\end{proof}

\begin{proof}[Proof of Theorem \ref{thm:mainresultgeomergod2}]
	Since $P$ and $P^*$ are uniformly ergodic, Proposition $7$ in \cite{RR04} gives
	\begin{align*}
		\sup_{x \in \mathcal{X}} ||P^n(x,\cdot) - \pi ||_{TV} &< \dfrac{1}{12}, \\
		\sup_{x \in \mathcal{X}} ||P^{*n}(x,\cdot) - \pi ||_{TV} &< \dfrac{1}{12},
	\end{align*}
	for all sufficiently large $n$. Therefore, for all sufficiently large $n$, Lemma \ref{lem:tvPtvM1M2} yields
	\begin{align*}
		\sup_{x \in \mathcal{X}} ||M_1(P^n)(x,\cdot) - \pi ||_{TV} &\leq 2 \sup_{x \in \mathcal{X}} ||P^n(x,\cdot) - \pi ||_{TV} + 2 \sup_{x \in \mathcal{X}} ||P^{*n}(x,\cdot) - \pi ||_{TV} < \dfrac{1}{3}, \\
		\sup_{x \in \mathcal{X}} ||M_2(P^n)(x,\cdot) - \pi ||_{TV} &\leq 3 \sup_{x \in \mathcal{X}} ||P^n(x,\cdot) - \pi ||_{TV} + 3 \sup_{x \in \mathcal{X}} ||P^{*n}(x,\cdot) - \pi ||_{TV} < \dfrac{1}{2}.
	\end{align*}
	Desired result follows from Proposition $7$ in \cite{RR04}.
\end{proof}

\begin{proof}[Proof of Corollary \ref{cor:mainresultgeomergod}]
	We follow a similar line of reasoning than in the proof of Theorem $12.3$ in \cite{LPW09}. For any $x,y \in \mathcal{X}$ and $n > t^*$, if $y \in A_{x,n}$, we have
	$$\bigg | \dfrac{P^n(x,y)}{\pi(y)} - 1 \bigg | = \bigg| \dfrac{M_2(P^n)(x,y)}{\pi(y)} - 1 \bigg | \leq \dfrac{e^{-n\gamma^{MH}}}{\pi_{min}},$$
	where the inequality follows from Theorem $12.3$ in \cite{LPW09}. Similarly,
	\begin{align*}
	\bigg | \dfrac{P^n(x,y)}{\pi(y)} - 1 \bigg | &=
	\bigg | \dfrac{M_1(P^n)(x,y)}{\pi(y)} - 1 \bigg | \leq \dfrac{e^{-n\gamma^{MH}}}{\pi_{min}} , \quad \text{if}\, y \in A_{x,n}^c \backslash\{x\},\\
	\bigg | \dfrac{P^n(x,y)}{\pi(y)} - 1 \bigg | &\leq  \dfrac{1}{2}\left | \dfrac{M_1(P^n)(x,y)}{\pi(y)} - 1 \right | + \dfrac{1}{2}\left | \dfrac{M_2(P^n)(x,y)}{\pi(y)} - 1 \right | \leq \dfrac{e^{-n\gamma^{MH}}}{\pi_{min}} , \quad \text{if}\, y = x.
	\end{align*}
	Lemma $6.13$ in \cite{LPW09} gives $||P^n(x,\cdot) - \pi ||_{TV} \leq \frac{e^{-n\gamma^{MH}}}{\pi_{min}}$, and the desired result follows from the definition of $t_{mix}$.
\end{proof}

\subsection{Examples}

We illustrate the usefulness of the MH-spectral gap using three examples. In the first two cases, both the additive reversiblization and multiplicative reversiblization fail to give insights on the total variation distance from stationarity, however the pseudo-spectral gap and MH-spectral gap can still provide informative bounds.

In the following examples, we will be calculating numerically $\gamma^{ps}$ and $\gamma^{MH}$ for several finite Markov chains. As these spectral gaps involve taking supremum over possibly countably infinite set, we employ a truncation procedure in the sense that we assume these gaps can be computed by, for large enough $N_0$,
\begin{align}
	\gamma^{ps} &= \max_{k \in \llbracket 1,N_0 \rrbracket} \{\gamma(P^{*k}P^k)/k\}, \label{eq:psgapprox} \\
	\gamma^{MH} &= 1 - \max_{k \in \mathcal{C} \cap \llbracket 1,N_0 \rrbracket}\{ |\lambda(M_2(P^k))|^{1/k}, \Lambda(M_1(P^k))^{1/k} \}. \label{eq:mhsgapprox}
\end{align}
Note that however we are not able to prove these two equations \eqref{eq:psgapprox} and \eqref{eq:mhsgapprox} or give bounds on $N_0$ in general. The above procedure is justified by the observation that for large enough $n$, the kernels $P^n, M_1(P^n)$ and $M_2(P^n)$ approaches $\Pi$, the Markov kernel with each row given by $\pi$, so the true value of these gaps can be found via searching in the interval $\llbracket 1,N_0 \rrbracket$ for large enough $N_0$.

\begin{example}[non-reversible walk on a triangle]\label{ex:1}
	The first example is taken from \cite{MT06} Example $5.2$. We consider a Markov chain on the triangle $\{0,1,2\}$ with transition probability given by $P(0,1) = P(1,2) = 1$ and $P(2,0) = P(2,1) = 0.5$. The stationary distribution is $\pi(0) = 0.2, \pi(1) = \pi(2) = 0.4$. The chain is non-reversible (for example, $P(1,0) = 0$, yet $P^*(1,0) = 1$), with eigenvalues $1, \frac{-1 \pm i}{2}$. The additive reversiblization bound does not work here, since the chain is not strongly aperiodic. For multiplicative reversiblization, it has been noted in \cite{MT06} that $\gamma(PP^*) = 0$, and the conductance bound does not work in this example as well.
	
	The classical bounds fail since the chain, if started at state $0$, requires two steps before its total variation distance decreases. Therefore, $\gamma^{ps}$ and $\gamma^{MH}$ are expected to give meaningful upper bounds in this case, since by definition they are catered to such situations. Indeed, finite calculations suggest that we take $N_0 = 100$ in \eqref{eq:psgapprox} and \eqref{eq:mhsgapprox}; if this is true, we have $\gamma^{ps} = \gamma(P^{*3}P^3)/3 = 0.25$, while $\gamma^{MH} = 1-\Lambda(M_1(P^6))^{1/6} = 0.151$. Comparing the results in Proposition $3.4$ in \cite{Paulin15} with Corollary \ref{cor:mainresultgeomergod}, we give a tighter upper bound in the total variation distance from stationarity, since the convergence rate is bounded by $||P^k(x,\cdot) - \pi||_{TV} \leq O((1-\gamma^{MH})^k) {\, = \,} O(0.849^k) \leq O((\sqrt{1-\gamma^{ps}})^k)  {\, = \,} O(\sqrt{0.75}^k)$.
	
	In the following, we justify $\gamma^{MH} = 1 - \Lambda(M_1(P^6))^{1/6}$ by assuming that we have an upper bound on the eigenvalues dynamics in \eqref{eq:eigenvalueub} below. While we are not able to prove this assumption of \eqref{eq:eigenvalueub}, the bound seems to be correct for finite $n$ as demonstrated in Figure \ref{fig:triangle}. We also collect a few useful properties about this three-state example:
	\begin{proposition}
		Suppose that $P$ is the three-state Markov kernel on a triangle as described above, that is,
		$$P = \begin{pmatrix} 
		0 & 1 & 0 \\ 
		0 & 0 & 1 \\ 
		0.5 & 0.5 & 0  
		\end{pmatrix}.$$		
		We have the following:
		\begin{enumerate}
			\item\label{it:triangle1} For $k \in \mathbb{N}$, $P^{4k} = (P^*)^{4k} = M_1(P^{4k}) = M_2(P^{4k}).$
			\item\label{it:triangle2} For $n \in \mathbb{N}$ with $n \geq 4$, both $M_1(P^n)$ and $M_2(P^n)$ are ergodic Markov kernel. Moreover, $\mathcal{C} = \{n \in \mathbb{N};~n \geq 4\}$.
			\item\label{it:triangle3} For $n \in \mathbb{N}$,
			\begin{align*}
				\Lambda(M_1(P^n)) &\geq \dfrac{\cos(3\pi n/4) + \cos(-3\pi n/4)}{2^{n/2+1}} \geq \lambda(M_2(P^n)).
			\end{align*}
			\item\label{it:triangle4} Assume that for $n \in \mathbb{N}$ with $n \geq 8$ we have the following upper bound:
			\begin{align}\label{eq:eigenvalueub}
				\max \{|\lambda(M_2(P^n))|, \Lambda(M_1(P^n))\} \leq \dfrac{2}{4^{\lfloor n/4\rfloor}}.
			\end{align}
			Consequently, 
			\begin{align*}
				\max_{n \geq 8} \{|\lambda(M_2(P^n))|^{1/n}, \Lambda(M_1(P^n))^{1/n}\} &\leq \max_{n \geq 8} \left(\dfrac{2}{4^{\lfloor n/4\rfloor}}\right)^{1/n} \leq \left(\dfrac{2}{4^2}\right)^{1/11}, \\
				\max_{n \in \llbracket 4,7 \rrbracket}\{ |\lambda(M_2(P^n))|^{1/n}, \Lambda(M_1(P^n))^{1/n} \} &= \Lambda(M_1(P^6))^{1/6} = \left(\dfrac{3}{8}\right)^{1/6} > \left(\dfrac{2}{4^2}\right)^{1/11}, \\
				\beta^{MH} &= \Lambda(M_1(P^6))^{1/6}.
			\end{align*}
		\end{enumerate}
	\end{proposition}
	\begin{proof}
		We first prove item \eqref{it:triangle1}. Note that
		$$P^4 = P^{*4} = \begin{pmatrix} 
		0 & 0.5 & 0.5 \\ 
		0.25 & 0.25 & 0.5 \\ 
		0.25 & 0.5 & 0.25  
		\end{pmatrix},$$
		and so $P^4$ is reversible. As a result, for $k \in \mathbb{N}$, $P^{4k} = (P^*)^{4k}$. In addition, by Lemma \ref{lem:M1M2prop} item \ref{state:3}, $M_1(P^{4k}) = M_2(P^{4k}) = P^{4k}$. Next, we prove item \eqref{it:triangle2}. For $n = 4$, we have
		\begin{align*}
		P^4 = M_1(P^4) = M_2(P^4) = \begin{pmatrix} 
		0 & 0.5 & 0.5 \\ 
		0.25 & 0.25 & 0.5 \\ 
		0.25 & 0.5 & 0.25  
		\end{pmatrix}
		\end{align*}
		For $n \geq 5$, we will prove by induction that the following holds for $x \in \{0,1,2\}$:
		\begin{align}
		&0 < P^n(x,0) \leq \dfrac{1}{4}, \quad 0 < P^n(x,1) \leq \dfrac{1}{2}, \quad 0 < P^n(x,2) \leq \dfrac{1}{2}, \label{eq:triangle1} \\
		&0 < P^{*n}(x,0) \leq \dfrac{1}{4}, \quad 0 < P^{*n}(x,1) \leq \dfrac{1}{2}, \quad 0 < P^{*n}(x,2) \leq \dfrac{1}{2}. \label{eq:triangle2}
		\end{align}
		This implies that $M_1(P^n)(x,y) > 0$ for all $x \neq y$ and $M_2(P^n)(0,0) \geq 0$, $M_2(P^n)(1,1), M_2(P^n)(2,2) > 0$, and so $M_1(P^n)$ and $M_2(P^n)$ are ergodic Markov kernel. To prove \eqref{eq:triangle1} and \eqref{eq:triangle2}, when $n = 5$ these holds since
		\begin{align*}
		P^5 = \begin{pmatrix} 
		0.25 & 0.25 & 0.5 \\ 
		0.25 & 0.5 & 0.25 \\ 
		0.125 & 0.375 & 0.5  
		\end{pmatrix}, \quad P^{*5} = \begin{pmatrix} 
		0.25 & 0.5 & 0.25 \\ 
		0.125 & 0.5 & 0.375 \\ 
		0.25 & 0.25 & 0.5  
		\end{pmatrix}.
		\end{align*}
		Assume that \eqref{eq:triangle1} and \eqref{eq:triangle2} hold for some $n$, using $P^{n+1} = P^n P$ leads to
		\begin{align*}
		&0 < P^{n+1}(x,0) = P^n(x,2)/2 \leq \dfrac{1}{4}, \\
		&0 < P^{n+1}(x,1) = P^n(x,0) + \dfrac{1}{2} P^n(x,2) \leq \dfrac{1}{2}, \\ 
		&0 < P^{n+1}(x,2) = P^n(x,1) \leq \dfrac{1}{2},  \\
		&0 < P^{*n+1}(x,0) = \dfrac{1}{2}P^{*n}(x,1) \leq \dfrac{1}{4}, \\
		&0 < P^{*n+1}(x,1) = P^{*n}(x,2) \leq \dfrac{1}{2}, \\
		&0 < P^{*n+1}(x,2) = P^{*n}(x,0) + \dfrac{1}{2} P^{*n}(x,1) \leq \dfrac{1}{2}. 
		\end{align*}
		Thus, $\{n \in \mathbb{N};~n \geq 4\} \subseteq \mathcal{C}$. To see that $\{1,2,3\} \notin \mathcal{C}$, we note that
		\begin{align*}
		M_1(P) &= \begin{pmatrix} 
		1 & 0 & 0 \\ 
		0 & 0.5 & 0.5 \\ 
		0 & 0.5 & 0.5  
		\end{pmatrix}, \quad M_2(P) = \begin{pmatrix} 
		-1 & 1 & 1 \\ 
		0.5 & -0.5 & 1 \\ 
		0.5 & 1 & -0.5  
		\end{pmatrix}, \\
		M_1(P^2) &= \begin{pmatrix} 
		1 & 0 & 0 \\ 
		0 & 1 & 0 \\ 
		0 & 0 & 1  
		\end{pmatrix}, \quad M_2(P^2) = \begin{pmatrix} 
			-1 & 1 & 1 \\ 
		0.5 & 0 & 0.5 \\ 
		0.5 & 0.5 & 0  
		\end{pmatrix}, \\
		M_1(P^3) &= \begin{pmatrix} 
		1 & 0 & 0 \\ 
		0 & 0.75 & 0.25 \\ 
		0 & 0.25 & 0.75  
		\end{pmatrix}, \quad M_2(P^3) = \begin{pmatrix} 
		0 & 0.5 & 0.5 \\ 
		0.25 & 0.25 & 0.5 \\ 
		0.25 & 0.5 & 0.25  
		\end{pmatrix}.
		\end{align*}
		Next, we prove item \eqref{it:triangle3}. By writing $\mathrm{Tr}(A)$ to be the trace of the matrix $A$, by Lemma \ref{lem:M1M2prop} item \ref{state:3} again we have
		$$2 + 4 \lambda_3(M_2(P^n)) \leq \mathrm{Tr}(P^n + P^{*n}) = \mathrm{Tr}(M_1(P^n) + M_2(P^n)) \leq 2 + 4 \lambda_2(M_1(P^n)),$$
		so it suffices to prove 
		$$\dfrac{\mathrm{Tr}(P^n + P^{*n}) - 2}{4} =  \dfrac{\cos(3\pi n/4) + \cos(-3\pi n/4)}{2^{n/2+1}}$$
		The eigenvalues of $P^n$ and $P^{*n}$ are $1, e^{\pm i 3\pi n/4}/2^{n/2}$. As the trace of $P^n + P^{*n}$ must be real, we have
		$$\mathrm{Tr}(P^n + P^{*n}) = 2 + 2 \mathrm{Re}(e^{i 3\pi n/4}/2^{n/2} + e^{-i 3\pi n/4}/2^{n/2}) = 2 + \dfrac{\cos(3\pi n/4) + \cos(-3\pi n/4)}{2^{n/2-1}},$$
		where we write $\mathrm{Re}(\cdot)$ to denote the real part. The desired result follows. Finally, we prove item \eqref{it:triangle4}, in which we only need to show for $n \geq 8$ we have
		$$\left(\dfrac{2}{4^{\lfloor n/4\rfloor}}\right)^{1/n} \leq \left(\dfrac{2}{4^2}\right)^{1/11}.$$
		Writing $n = 4l + j$ for $l \in \mathbb{N}$, $l \geq 2$ and $j \in \{0,1,2,3\}$. Let
		\begin{align*}
			g(l) :&= \left(\dfrac{2}{4^l}\right)^{1/(4l+j)}, \\
			\dfrac{d}{dl} \ln g(l) &= \dfrac{-j \ln 4 - 4 \ln 2}{(4l+j)^2} < 0.
		\end{align*}
		As a result $g$ is decreasing in $l$ and hence the maximizer of $g$ occurs at $l = 2, j = 3, n = 11$.
	\end{proof}
	
\begin{figure}[H]
	\centering
	\begin{subfigure}{0.5\textwidth}
		\centering
		\includegraphics[scale=0.5]{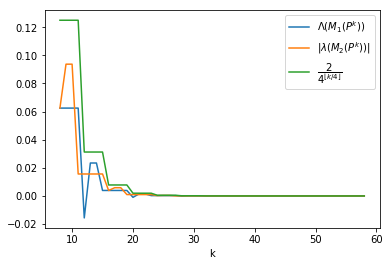}
		\caption{Plot of $\Lambda(M_1(P^k)), |\lambda(M_2(P^k))|, \left(\frac{2}{4^{\lfloor k/4\rfloor}}\right)$ against $k$}
	\end{subfigure}%
	\begin{subfigure}{0.5\textwidth}
		\centering
		\includegraphics[scale=0.5]{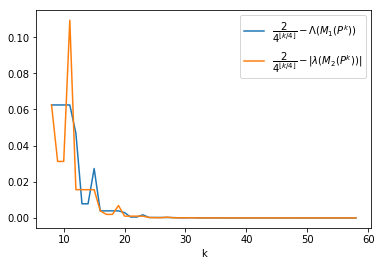}
		\caption{Plot of $\left(\frac{2}{4^{\lfloor k/4\rfloor}}\right) - \Lambda(M_1(P^k)), \left(\frac{2}{4^{\lfloor k/4\rfloor}}\right) - |\lambda(M_2(P^k))|$ against $k$. The values are all non-negative.}
	\end{subfigure}
	\caption{Non-reversible walk on a triangle}
	\label{fig:triangle}
\end{figure}
\end{example}

\begin{example}[non-reversible Markov chain sampler]\label{ex:2}
	The second example is taken from \cite{MT06} Example $5.3$ and \cite{DHN00}. Consider a Markov chain on $\mathcal{X} = \mathbb{Z}/2m\mathbb{Z}$ labeled by $\{-(m-1),\ldots,0,1,\ldots,m\}$, with transitions $P(i,i+1) = 1 - \frac{1}{m}, P(i,-i) = \frac{1}{m}$. The chain is doubly stochastic with stationary distribution being the uniform distribution on the state space. It is shown in Theorem $1$ of \cite{DHN00} that $t_{mix}(\epsilon) = \Theta(m \log (1/\epsilon))$, and in \cite{MT06} that existing upper bounds cannot provide useful information.
	
	We now fix $m = 3$. Finite calculations suggest that we can consider $N_0 = 100$ in \eqref{eq:psgapprox} and \eqref{eq:mhsgapprox}; if this is true, we demonstrate that $\gamma^{ps}$ and $\gamma^{MH}$ both give meaningful bounds. By computation, we have $\gamma^{ps} {\, = \,} \gamma(P^{*3}P^3)/3 = 0.315$, and $\gamma^{MH} {\, = \,} 1-|\lambda(M_2(P^4))|^{1/4} = 0.270$. Similar to Example \ref{ex:1}, the upper bound provided by Corollary \ref{cor:mainresultgeomergod} outperforms that in \cite{Paulin15}, since $||P^k(x,\cdot) - \pi||_{TV} \leq O((1-\gamma^{MH})^k) {\, = \,} O(0.730^k) \leq O((\sqrt{1-\gamma^{ps}})^k) {\, = \,} O(\sqrt{0.685}^k)$.
	
	Next, we consider an example of larger scale and we take $m = 100$. Finite calculations again suggest that we can consider $N_0 = 500$ in \eqref{eq:psgapprox} and \eqref{eq:mhsgapprox}. For illustration purposes, we plot $1-\max \{ |\lambda(M_2(P^k))|^{1/k}, \Lambda(M_1(P^k))^{1/k} \}$ and $\gamma(P^{*k}P^k)/k$ against $k$ in Figure \ref{fig:NRMS}. By computation, we see that $\beta^{MH} {\, = \,} 0.999914$ and $\gamma^{ps} {\, = \,} 0.008671$. In this case, the pseudo-spectral gap performs better than that of MH-spectral gap since
	$||P^k(x,\cdot) - \pi||_{TV} \leq O((\sqrt{1-\gamma^{ps}})^k) {\, = \,} O(0.995655^k) \leq O((1-\gamma^{MH})^k) {\, = \,} O(0.999914^k)$. Both bounds are asymptotically tighter than the bound obtained in \cite[Theorem $1$]{DHN00}, which gives $||P^k(x,\cdot) - \pi||_{TV} \leq (1-2^{-7})^{\lfloor k/400 \rfloor}$.
\end{example}

\begin{figure}[H]
	\centering
	\begin{subfigure}{0.5\textwidth}
		\centering
		\includegraphics[scale=0.4]{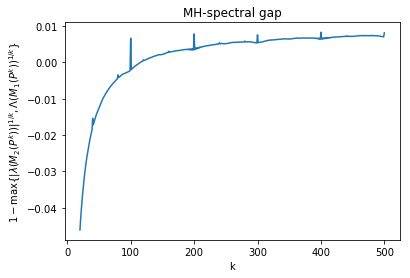}
		\caption{Plot of $1-\max \{ |\lambda(M_2(P^k))|^{1/k}, \Lambda(M_1(P^k))^{1/k} \}$ against $k$}
	\end{subfigure}%
	\begin{subfigure}{0.5\textwidth}
		\centering
		\includegraphics[scale=0.4]{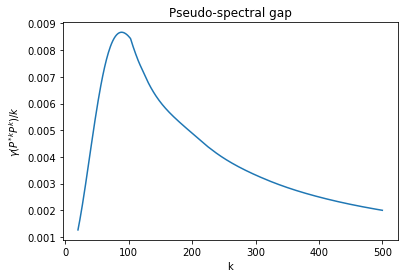}
		\caption{Plot of $\gamma(P^{*k}P^k)/k$ against $k$}
	\end{subfigure}
	\caption{Non-reversible Markov chain sampler with $m = 100$ and $N_0 = 500$}
	\label{fig:NRMS}
\end{figure}

\begin{example}[winning streak]\label{ex:3}
	The third example is the so-called winning streak Markov chain. It has been studied in Example $4.1.5$ and Section $5.3.5$ in \cite{LPW09}. Consider a Markov chain on $\mathcal{X} = \{0,\ldots,m\}$ with transitions $P(i,0) = P(i,i+1) = P(m,m) = 1/2$. One remarkable property of such a chain is that its time-reversal, $P^*$, attains \textit{exactly} the stationary distribution in $m$ steps.
	
	By a coupling argument, $d(n) \leq \frac{1}{2^n}$ for all $m$. Yet, for $P^*$, its mixing time is of order $m$. For now we fix $m = 4$, and finite calculations suggest that we can consider $N_0 = 100$ in \eqref{eq:psgapprox} and \eqref{eq:mhsgapprox}; if this is true, we have $\gamma^{ps} {\, = \,} \gamma(PP^*) = 0.5$ and $\gamma^{MH} {\, = \,} 0.138$, so both the multiplicative reversiblization and pseudo-spectral gap give a correct order of convergence rate. The performance of $\gamma^{MH}$ is poor in this example, due to the fact that $P^*$ has a much slower mixing time when compared to $P$.
	
	Next, we consider an example of larger scale and we take $m = 50$. Finite calculations again suggest that we can consider $N_0 = 100$ in \eqref{eq:psgapprox} and \eqref{eq:mhsgapprox}. For illustration purposes, we again plot $1-\max \{ |\lambda(M_2(P^k))|^{1/k}, \Lambda(M_1(P^k))^{1/k} \}$ and $\gamma(P^{*k}P^k)/k$ against $k$ in Figure \ref{fig:ws}. An interesting feature of the MH-spectral gap plot Figure \ref{fig:ws}(a) is that $1-\max \{ |\lambda(M_2(P^k))|^{1/k}, \Lambda(M_1(P^k))^{1/k} \}$ rises abruptly to $1$ when $k = m = 50$, which should be the case as the chain attains exactly $\pi$ in $m = 50$ steps. Such pattern however is missing in the graph of Figure \ref{fig:ws}(b). By computation, we see that $\beta^{MH} {\, = \,} 0.9999$ and $\gamma^{ps} {\, = \,} 0.4961$. In this case, the pseudo-spectral gap performs significantly better than that of MH-spectral gap since
	$||P^k(x,\cdot) - \pi||_{TV} \leq O((\sqrt{1-\gamma^{ps}})^k) {\, = \,} O(0.7099^k) \leq O((1-\gamma^{MH})^k) {\, = \,} O(0.9999^k)$. Both bounds are asymptotically weaker than the coupling bound as described in the second paragraph above.
\end{example}

\begin{figure}[H]
	\centering
	\begin{subfigure}{0.5\textwidth}
		\centering
		\includegraphics[scale=0.4]{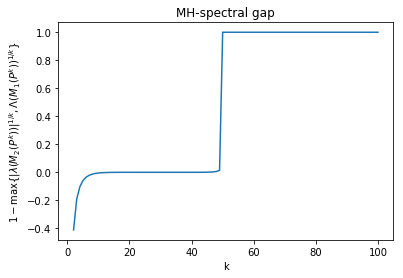}
		\caption{Plot of $1-\max \{ |\lambda(M_2(P^k))|^{1/k}, \Lambda(M_1(P^k))^{1/k} \}$ against $k$}
	\end{subfigure}%
	\begin{subfigure}{0.5\textwidth}
		\centering
		\includegraphics[scale=0.4]{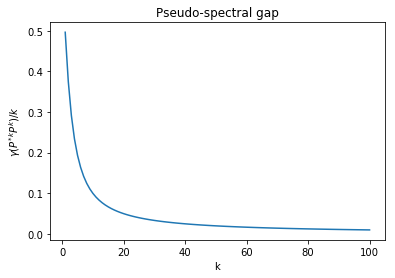}
		\caption{Plot of $\gamma(P^{*k}P^k)/k$ against $k$}
	\end{subfigure}
	\caption{Winning streak Markov chain with $m = 50$ and $N_0 = 100$}
	\label{fig:ws}
\end{figure}

\section{Variance bounds}\label{sec:varbd}

In this section, we prove variance bounds for Markov chains in terms of the MH-spectral gap. The readers should compare Theorem \ref{thm:varbd} with Lemma $12.20$ in \cite{LPW09} and Theorem $3.5$,$3.7$ in \cite{Paulin15}.


\begin{theorem}\label{thm:varbd}
	Let $(X_n)_{n \geq 0}$ be a Markov chain with Markov kernel $P$, stationary measure $\pi$ and MH-spectral gap $\gamma^{MH}$. Suppose that $f \in L^2(\pi)$, and define the variance and asymptotic variance to be respectively
	\begin{align*}
		V_f &:= \mathrm{Var}_{\pi}(f), \\
		\sigma_{as}^2 &:= \lim_{n \to \infty} \dfrac{1}{n} \mathrm{Var}_{\pi} \left(\sum_{i=1}^n f(X_i) \right).
	\end{align*}
	The variance bounds are given by
	\begin{align}
		\mathrm{Var}_{\pi} \left(\sum_{i=1}^n f(X_i) \right) &\leq n V_f \left(|\mathcal{C}^c|+ \dfrac{2}{\gamma^{MH}} \right), \label{eq:varbd} \\
		\bigg| \mathrm{Var}_{\pi} \left(\sum_{i=1}^n f(X_i) \right) - n \sigma_{as}^2 \bigg| &\leq 4 V_f \left(1 + |\mathcal{C}^c| + \dfrac{4 (\beta^{MH})^{|\mathcal{C}^c| +1}}{\gamma^{MH}} \right)^2. \label{eq:varbdas}
	\end{align}
	More generally, if $f_i \in L^2(\pi)$ for $i=1,\ldots,n$, then
	\begin{align}
		\mathrm{Var}_{\pi} \left(\sum_{i=1}^n f_i(X_i) \right) &\leq \sum_{i=1}^n \mathrm{Var}_{\pi}(f_i(X_i)) \left(|\mathcal{C}^c|+ \dfrac{2}{\gamma^{MH}} \right). \label{eq:varbdg}
	\end{align}
\end{theorem}


Before we give the proof of Theorem \ref{thm:varbd}, we state a lemma that bounds the operator norm of $P$ by that of $M_1$ and $M_2$ in $L^2_0(\pi)$.

\begin{lemma}\label{lem:Popnbd}
	Suppose that $P$ is a Markov kernel with stationary measure $\pi$. Then
	$$||P||_{L^2_0(\pi) \to L^2_0(\pi)} \leq ||M_1||_{L^2_0(\pi) \to L^2_0(\pi)} + ||M_2||_{L^2_0(\pi) \to L^2_0(\pi)} + |\lambda(M_1(P^2))|^{1/2} + |\lambda(M_2(P^2))|^{1/2} .$$
\end{lemma}

\begin{proof}
	By Lemma \ref{lem:M1M2prop}\ref{state:1}, we have $||(P+P^*)f||_{\pi}^2 = ||(M_1+M_2)f||_{\pi}^2$ where $f \in L^2_0(\pi)$. Rearranging the terms give
	\begin{align*}
		\langle Pf,Pf \rangle_{\pi} &= \langle M_1f,M_1f \rangle_{\pi} + \langle M_2f,M_2f \rangle_{\pi} + \langle M_1f,M_2f \rangle_{\pi} + \langle M_2f,M_1f \rangle_{\pi} \\
		&\quad - \langle P^*f,P^*f \rangle_{\pi} - \langle f,(P^*)^2f \rangle_{\pi} - \langle f,P^2f \rangle_{\pi} \\
		&\leq \langle M_1f,M_1f \rangle_{\pi} + \langle M_2f,M_2f \rangle_{\pi} + \langle M_1f,M_2f \rangle_{\pi} + \langle M_2f,M_1f \rangle_{\pi} \\
		&\quad - \langle f,M_1(P^2)f \rangle_{\pi} - \langle f,M_2(P^2)f \rangle_{\pi} ,
	\end{align*}
	where we used that $P^2 + (P^*)^2 = M_1(P^2) + M_2(P^2)$ by Lemma \ref{lem:M1M2prop}\ref{state:1} and $\langle P^*f,P^*f \rangle_{\pi} \geq 0$ in the inequality. Therefore, we have
	$$||Pf||_{\pi} \leq  (||M_1||_{L^2_0(\pi) \to L^2_0(\pi)} + ||M_2||_{L^2_0(\pi) \to L^2_0(\pi)}) ||f||_{\pi} + |\lambda(M_1(P^2))|^{1/2} + |\lambda(M_2(P^2))|^{1/2} .$$
	Result follows by taking supremum over all $f$ with $||f||_{\pi} \leq 1$ and $\E_{\pi} f = 0$.	
\end{proof}

\begin{proof}[Proof of Theorem \ref{thm:varbd}]
	Assume without loss of generality that $\E_{\pi}(f) = 0$ and $\E_{\pi}(f_i) = 0$ for $i=1,\ldots,n$. We first show \eqref{eq:varbd}. Since $X_0 \sim \pi$ and by Lemma \ref{lem:peskunpositive},
	$$\E_{\pi}(f(X_i)f(X_j)) = \langle f, P^{|j-i|} f \rangle_{\pi} \leq \langle f, M_1(P^{|j-i|}) f \rangle_{\pi} = \langle f, (M_1(P^{|j-i|}) - \pi) f \rangle_{\pi}.$$
	Summing up $j$ from $1$ to $n$ leads to
	\begin{align*}
		\E_{\pi}\left(f(X_i) \sum_{j=1}^n f(X_j) \right) &\leq \sum_{j=1}^n \langle f, (M_1(P^{|j-i|}) - \pi) f \rangle_{\pi} \\
		&\leq \E_{\pi}(f^2) \sum_{j=1}^n ||M_1(P^{|j-i|})||_{L^2_0(\pi) \to L^2_0(\pi)} \\
		&\leq V_f \left( |\mathcal{C}^c| + \sum_{j=1}^n (\beta^{MH})^{|j-i|} \right) \\
		&\leq V_f \left( |\mathcal{C}^c| +\dfrac{2}{\gamma^{MH}} \right).
	\end{align*}
	\eqref{eq:varbd} follows when we sum $i$ from $1$ to $n$. Next, to show \eqref{eq:varbdg}, we observe that
	\begin{align*}
		\E_{\pi}(f_i(X_i)f_j(X_j)) &= \langle f_i, P^{|j-i|} f_j \rangle_{\pi} \\
		&\leq \langle f_i, (M_1(P^{|j-i|}) - \pi) f_j \rangle_{\pi} \\
		&\leq ||f_i||_{\pi} ||f_j||_{\pi} ||M_1(P^{|j-i|})||_{L^2_0(\pi) \to L^2_0(\pi)} \\
		&\leq \dfrac{1}{2} (\E_{\pi}f_i^2 + \E_{\pi}f_j^2) ||M_1(P^{|j-i|})||_{L^2_0(\pi) \to L^2_0(\pi)}.
	\end{align*}
	\eqref{eq:varbdg} follows when we sum $i,j$ from $1$ to $n$, and $\sum_{j=1}^n ||M_1(P^{|j-i|})||_{L^2_0(\pi) \to L^2_0(\pi)} \leq |\mathcal{C}^c| +\dfrac{2}{\gamma^{MH}}$. Finally, we will show \eqref{eq:varbdas}. Following from the proof of Theorem $3.5$ and $3.7$ in \cite{Paulin15}, using the definition of $\sigma_{as}^2$, we have
	$$\bigg| \mathrm{Var}_{\pi} \left(\sum_{i=1}^n f(X_i) \right) - n \sigma_{as}^2 \bigg| = | \langle f, 2 (I - (P - \pi)^{n-1})(I - (P - \pi))^{-2}f \rangle_{\pi} | .$$
	Note that $||I - (P - \pi)^{n-1}||_{L^2(\pi) \to L^2(\pi)} \leq 2$, and by Lemma \ref{lem:Popnbd},
	\begin{align*}
		||(I - (P - \pi))^{-1}||_{L^2(\pi) \to L^2(\pi)} &\leq \sum_{k=0}^{\infty} ||(P - \pi)^{k}||_{L^2(\pi) \to L^2(\pi)} \\
		&\leq 1 + |\mathcal{C}^c| + \sum_{k=|\mathcal{C}^c| +1}^{\infty} ||P^{k}||_{L^2_0(\pi) \to L^2_0(\pi)} \\
		&\leq 1 + |\mathcal{C}^c| + 4 \sum_{k=|\mathcal{C}^c| +1}^{\infty} (\beta^{MH})^{k} \\
		&=  1 + |\mathcal{C}^c| + \dfrac{4 (\beta^{MH})^{|\mathcal{C}^c| +1}}{\gamma^{MH}}.
	\end{align*}
\end{proof}

\section{Metastability, conductance and Cheeger's inequality}\label{sec:metastable}

In this section, we aim at analyzing metastability, conductance, Cheeger's inequality and their relationships with the two MH kernels. We begin by briefly recalling these concepts.

\begin{definition}[Metastability of a set]\label{def:metastability}
	Let $A,B \in \mathcal{F}$ be measurable subsets of $\mathcal{X}$. Denote by
	$$Q(A,B) := \dfrac{1}{\pi(A)} \int_A p(x,B) \, \pi(dx) = \dfrac{\langle P \1_{A}, \1_{B} \rangle_{\pi}}{\langle \1_{A}, \1_{A} \rangle_{\pi}}\,,$$
	if $\pi(A) > 0$ and $0$ otherwise. $A$ is said to be \textit{metastable} (resp.~\textit{invariant}) if 
	$$Q(A,A) \approx 1 \, (\mathrm{resp}.~Q(A,A) = 1)\,.$$	
\end{definition}


\begin{rk}
	Denote $A^c$ to be the complement of $A \subset \mathcal{X}$, then $Q(A,A^c)$ is also known as the conductance of the set $A$. See Definition \ref{def:conductance} below.
\end{rk}

Note that metastability means ``almost invariant", in the sense that $Q(A,A)$ is close to $1$. In reality, we are more interested in measuring the metastability of an arbitrary partition of the state space $\mathcal{X}$, in which we state in the following:

\begin{definition}[Metastability of a partition]
	Suppose that $\mathcal{D} = \{A_1,\ldots,A_n\}$ is a partition of $\mathcal{X}$. The metastability of $\mathcal{D}$ is denoted by
	$$m(\mathcal{D}) := \sum_{i=1}^n Q(A_i,A_i)\,.$$
	$\mathcal{D}$ is said to be metastable if $m(\mathcal{D}) \approx n$.
\end{definition}

The next definition measures the ``leakage" of a set $A$ at time $t$, which is first introduced by \cite{D82, S84}.

\begin{definition}[Leakage]
	The leakage of a set $A \in \mathcal{F}$ at time $t$ is denoted by
	$$l(A,t) := \dfrac{||\1_{A} - P^t \1_{A}||_{L^1(\pi)}}{2\pi(A)(1-\pi(A))}\,.$$
	This can be rewritten as
	$$l(A,t) = \int_{\mathcal{X}\backslash A} \left(P^t \dfrac{\1_A}{\pi(A)}\right)(x) \dfrac{\pi(dx)}{1-\pi(A)}\,,$$
	measuring the probability of $\1_A/\pi(A)$ being outside $A$ at time $t$.
\end{definition}

\begin{rk}
	Another related measure of bottleneckness is conductance. See Definition \ref{def:conductance} below.
\end{rk}

Next, we introduce a key assumption (see e.g. \cite{D82, S84, HS06}) that will be used in subsequent sections:

\begin{assumption}\label{assumption}
	Suppose that $Q : L^2(\pi) \rightarrow L^2(\pi)$ is a self-adjoint Markov kernel with $n$  eigenvalues denoted by
	$$1 = \lambda_1 > \lambda_2 \geq \ldots \geq \lambda_n\,.$$
	In addition, the spectrum $\sigma(Q)$ of $Q$ is contained in
	$$\sigma(Q) \subset [a,b] \cup \{\lambda_n,\ldots,\lambda_1\}\,,$$
	where $-1 < a \leq b < \lambda_n$. In this sense, the eigenvalues $\lambda_1,\ldots,\lambda_n$ are called \textit{dominant} as they are larger than $b$.
\end{assumption}

If $Q$ is a finite Markov chain, or if $Q$ is geometrically ergodic, or if $Q$ is $V$-uniformly ergodic, then it can be shown that $Q$ satisfies Assumption \ref{assumption}, see e.g. \cite{HS06, SS13}. Under Assumption \ref{assumption} with $n = 2$, if the eigenvalue $\lambda_2$ is ``close" to $1$, then this is known as ``almost degeneracy", which allows us to partition $\mathcal{X}$ into two metastable regions. This has been the subject of investigation in \cite{D82, S84}.


Next, we provide a quick review on the notion of conductance and Cheeger's inequality, which are first introduced to the Markov chain literature in \cite{DS91}.

\begin{definition}[Conductance]\label{def:conductance}
	The conductance of the set $A$ is
	$$\Phi(A) := Q(A,A^c) = 1 - Q(A,A)\,,$$
	where $Q(A,A^c)$ is defined in Definition \ref{def:metastability}. The conductance of the chain is defined to be
	$$\Phi_*(2) := \min_{A \neq \emptyset,\mathcal{X}} \max \{\Phi(A),\Phi(A^c)\} = \min_{A: 0 < \pi(A) \leq 1/2} \Phi(A)\,.$$
	For $k \in \mathbb{N}$, let $\mathcal{D}_k = \{A_1,\ldots,A_k\}$ be the set of $k$-uples of disjoint and $\pi$-non-negligible subsets of $\mathcal{X}$. Then the $k$-way expansion is 
	$$\Phi_*(k) := \min_{(A_1,\ldots,A_k) \in \mathcal{D}_k} \max_{i \in \llbracket k \rrbracket} \Phi(A_i)\,.$$
\end{definition}

Next, we recall the Cheeger's inequality and its higher-order variants, which provide a two-sided bound in the spectral gap in terms of the $k$-way expansion, see e.g. \cite{LGT12} and \cite[Proposition $5$]{M15} .

\begin{theorem}[Higher-order Cheeger's inequality]\label{thm:revCheeger}
	Suppose that $P$ is the Markov kernel of a discrete-time reversible finite Markov chain with eigenvalues $1 = \lambda_1 \geq \ldots \geq \lambda_n$. For $k \in \llbracket n \rrbracket$,
	$$\dfrac{1-\lambda_k}{2} \leq \Phi_*(k) \leq O(k^4) \sqrt{1-\lambda_k}\,.$$
\end{theorem}

\subsection{Main results}

In this section, we demonstrate that, by means of comparison (i.e. Peskun's ordering as in Lemma \ref{lem:peskunpositive}), that existing results on metastability, leakage and Cheeger's inequality can be readily extended to the non-reversible case. We  first give spectral bounds on the metastability of partition, in terms of spectral objects associated with the first and second MH kernel $M_1$ and $M_2$.

\begin{theorem}[Metastability]\label{thm:mainMH}
	Suppose that $P$ is the Markov kernel of a non-reversible Markov chain on $\mathcal{X}$, with the first and second MH kernel denoted by $M_1$ and $M_2$ respectively. In addition, for $i = 1,2$, $M_i$ satisfies Assumption \ref{assumption} with dominant eigenvalues-eigenvectors denoted by $(\lambda_j^{(i)}, \phi_j^{(i)})_{j=1}^n$. For any partition $\mathcal{D} = \{A_1,\ldots,A_n\}$, the metastability of $\mathcal{D}$ is bounded by 
	\begin{align}\label{eq:mainMH}
	1 + \sum_{j=2}^n \rho_j \lambda_j^{(2)} + c \leq m(\mathcal{D}) \leq 
	1 + \sum_{j=2}^n \lambda_j^{(1)}\,,
	\end{align}
	where $\Gamma$ is the orthogonal projection from $L^2(\pi) \to \mathrm{Span}\{\1_{A_1},\ldots,\1_{A_n}\}$, $\rho_j := ||\Gamma \phi_j^{(2)}||_{\pi}^2 \in [0,1]$ for $j = 2,\ldots,n$, and 
	$$c := a \left(\sum_{j=2}^n 1-\rho_j \right)\,,$$
	with $a$ being defined in Assumption \ref{assumption} for $M_2$.
\end{theorem}

\begin{rk}
	This result should be compared with \cite[Theorem $2$]{HS06} and \cite[Theorem $5.8$]{SS13}, in which we retrieve the corresponding results since $P = M_1 = M_2$ and hence $\lambda_j = \lambda_j^{(1)} = \lambda_j^{(2)}$ in the reversible case. 
\end{rk}

The next theorem gives spectral bounds on leakage:

\begin{theorem}[Leakage]\label{thm:mainMH2}
	Suppose that $P$ is the Markov kernel of a non-reversible Markov chain on $\mathcal{X}$, with the first and second MH kernel denoted by $M_1$ and $M_2$ respectively. In addition, for $i = 1,2$ and $t \in \mathbb{N}$, $M_i(P^t)$ satisfies Assumption \ref{assumption} with $n=2$ and dominant eigenvalues-eigenvectors denoted by $(\lambda_j^{(i)}(P^t), \phi_j^{(i)}(P^t))_{j=1}^2$. If $\{A,B\}$ is a partition of $\mathcal{X}$, then for all $t \in \mathbb{N}$, 
	\begin{align*}
	1 - \lambda_2^{(1)}(P^t) \leq l(A,t) \leq 1 - \gamma_A^2(M_2(P^t)) \lambda_2^{(2)}(P^t)\,,
	\end{align*}
	where $\gamma_A(M_i(P^t)) = \langle \psi_A, \phi_2^{(i)}(P^t) \rangle_{\pi}$ for $i = 1,2$ and 
	$$\psi_A = \sqrt{\dfrac{\pi(B)}{\pi(A)}} \1_A - \sqrt{\dfrac{\pi(A)}{\pi(B)}} \1_B\,.$$
\end{theorem}

\begin{rk}
	This result should be compared with \cite[Theorem $5$]{S84}, in which we retrieve the corresponding upper bound in the reversible case since $P = M_1 = M_2$ and hence $\lambda_j = \lambda_j^{(1)} = \lambda_j^{(2)}$.
\end{rk}

Finally, we give a version of Cheeger's inequality in bounding the $k$-way expansion, in terms of the eigenvalues of the two Metropolis kernels. This result should be compared with Theorem \ref{thm:revCheeger}.

\begin{theorem}[Cheeger's inequality]\label{thm:Cheeger}
	Suppose that $P$ is the Markov kernel of a non-reversible Markov chain on a finite state space $\mathcal{X}$, with the first and second MH kernel denoted by $M_i$ and eigenvalues $(\lambda_j^{(i)})_{j=1}^n$ for $i=1,2$. For $k \in \llbracket n \rrbracket$,
	\begin{align}\label{eq:Cheeger}
	\dfrac{1-\lambda_k^{(1)}}{2} \leq \Phi_*(k) \leq O(k^4) \sqrt{1-\lambda_k^{(2)}}\,.
	\end{align}
\end{theorem}

\subsection{Proofs}\label{sec:proof}

\begin{proof}[Proof of Theorem \ref{thm:mainMH}]
	The key step is the Peskun ordering between $P, M_1, M_2$, which yields, for any $f \in L^2(\pi)$, the following inequalities:
	\begin{equation}\label{eq:Peskun}
	\langle M_2 f,f \rangle_{\pi} \leq \langle Pf,f\rangle_{\pi} \leq \langle M_1f,f\rangle_{\pi}\,.
	\end{equation}
	This allows us to link $m(\mathcal{D})$ to the eigenvalues of $M_1$ and $M_2$. More precisely, we first show the upper bound of \eqref{eq:mainMH}. Making use of the definition, we have
	\begin{align*}
	m(\mathcal{D}) = \sum_{i=1}^n \dfrac{\langle P \1_{A_i}, \1_{A_i} \rangle_{\pi}}{\langle \1_{A_i}, \1_{A_i} \rangle_{\pi}} \leq \sum_{i=1}^n \dfrac{\langle M_1 \1_{A_i}, \1_{A_i} \rangle_{\pi}}{\langle \1_{A_i}, \1_{A_i} \rangle_{\pi}} \leq  1 + \sum_{j=2}^n \lambda_j^{(1)}\,,
	\end{align*}
	where the first inequality follows from \eqref{eq:Peskun} with $f = \1_{A_i}$, and we use \cite[Theorem $2$]{HS06} in the second inequality since $M_1$ is a self-adjoint Markov kernel satisfying Assumption \ref{assumption}. Next, to show the lower bound, using \eqref{eq:Peskun} again, we arrive at
	\begin{align*}
	m(\mathcal{D}) \geq \sum_{i=1}^n \langle M_{2} \chi_{A_i}, \chi_{A_i} \rangle_{\pi}\,,
	\end{align*}
	where $\chi_{A_i} := \frac{\1_{A_i}}{\sqrt{\langle \1_{A_i}, \1_{A_i} \rangle_{\pi}}}$. The remaining part of the proof follows a similar argument as in \cite{HS06}. Denote the orthogonal projection by $\Pi : L^2(\pi) \rightarrow \mathrm{Span}\{\phi^{(2)}_1,\ldots,\phi^{(2)}_n\}$ and its orthogonal complement by $\Pi^{\perp} = I - \Pi$. Note that
	\begin{align*}
	\sum_{i=1}^n \langle M_{2} \chi_{A_i}, \chi_{A_i} \rangle_{\pi} &= \sum_{i=1}^n \langle (M_{2} -a I)(\Pi + \Pi^{\perp}) \chi_{A_i}, \chi_{A_i} \rangle_{\pi} + a \sum_{i=1}^n \langle \chi_{A_i}, \chi_{A_i} \rangle_{\pi} \\
	&= \sum_{i=1}^n \langle (M_{2} -a I)\Pi \chi_{A_i}, \Pi \chi_{A_i} \rangle_{\pi} + \sum_{i=1}^n \langle (M_{2} -a I)\Pi^{\perp} \chi_{A_i}, \Pi^{\perp} \chi_{A_i} \rangle_{\pi} + a n \\
	&\geq \sum_{i=1}^n \langle (M_{2} -a I)\Pi \chi_{A_i}, \Pi \chi_{A_i} \rangle_{\pi} + a n \\
	&= \sum_{i=1}^n \bigg\langle \sum_{k=1}^n (\lambda_k^{(2)} - a) \langle \chi_{A_i},\phi_k^{(2)} \rangle_{\pi} \phi_k^{(2)}, \sum_{k=1}^n \langle \chi_{A_i},\phi_k^{(2)} \rangle_{\pi} \phi_k^{(2)} \bigg\rangle_{\pi} + a n \\
	&= \sum_{i=1}^n \sum_{k=1}^n (\lambda_k^{(2)} - a) \langle \chi_{A_i},\phi_k^{(2)} \rangle_{\pi}^2 + a n = \sum_{k=1}^n (\lambda_k^2 - a) ||\Gamma \phi_k^{(2)}||_{\pi}^2 + a n = 1 + \sum_{j=2}^n \rho_j \lambda_j^{(2)} + c\,,
	\end{align*}
	where the inequality follows from the fact that $(M_2 - aI)$ is self-adjoint and positive semidefinite, the fourth equality comes from the fact that the set $\{\phi^{(2)}_1,\ldots,\phi^{(2)}_n\}$ is orthonormal, the fifth equality makes use of Parseval's identity, and we use $\lambda_1^{(2)} = 1, \phi_1^{(2)} = \1, ||\Gamma \phi_1^{(2)}||_{\pi}^2 = 1$ in the last equality. The desired result follows.	
\end{proof}

\begin{proof}[Proof of Theorem \ref{thm:mainMH2}]
	Similar to the proof of Theorem \ref{thm:mainMH}, the crux again lies at the appropriate use of \eqref{eq:Peskun}. First, by \cite[Lemma $4$]{S84} and \eqref{eq:Peskun}, we have 
	\begin{align*}
	\langle (I - M_1(P^t)) \psi_A, \psi_A \rangle_{\pi} \leq \langle (I - P^t) \psi_A, \psi_A \rangle_{\pi} = l(A,t) \leq
	\langle (I - M_2(P^t)) \psi_A, \psi_A \rangle_{\pi}\,,
	\end{align*}
	so it suffices to show that 
	\begin{align}
	\langle (I - M_2(P^t)) \psi_A, \psi_A \rangle_{\pi} &\leq 1 - \gamma_A^2(M_2(P^t)) \lambda_2^{(2)}(P^t)\,, \label{eq:1} \\
	\langle (I - M_1(P^t)) \psi_A, \psi_A \rangle_{\pi} &\geq 1 - \lambda_2^{(1)}(P^t)\,. \label{eq:2}
	\end{align}
	The rest of the proof is similar to that of \cite[Theorem $5$]{S84}. For $i = 1,2$ and $t \in \mathbb{N}$, denote by $\Pi_i(P^t)$ to be the orthogonal projection $\Pi_i(P^t) : L^2(\pi) \rightarrow \mathrm{Span}\{\phi^{(i)}_1(P^t),\phi^{(i)}_2(P^t)\}$ and its orthogonal complement by $\Pi^{\perp}_i(P^t) = I - \Pi_i(P^t)$. Since $\psi_A$ is orthogonal to $\phi_1^{(i)}(P^t) = \1$, we have
	\begin{align}\label{eq:3}
	\Pi_i(P^t) \psi_A = \langle \psi_A,\phi_2^{(i)}(P^t) \rangle_{\pi} \phi_2^{(i)}(P^t) = \gamma_A(M_i(P^t)) \phi_2^{(i)}(P^t)\,.
	\end{align}
	We proceed to show \eqref{eq:1}. Note that
	\begin{align*}
	\langle (I - M_2(P^t)) \psi_A, \psi_A \rangle_{\pi} &= \langle (I - M_2(P^t))(\Pi_2(P^t)+\Pi^{\perp}_2(P^t)) \psi_A, \psi_A \rangle_{\pi} \\
	&= \gamma_A^2(M_2(P^t))(1-\lambda_2^{(2)}(P^t)) + \langle (I - M_2(P^t))\Pi^{\perp}_2(P^t) \psi_A, \psi_A \rangle_{\pi} \\
	&\leq \gamma_A^2(M_2(P^t))(1-\lambda_2^{(2)}(P^t)) + \langle \Pi^{\perp}_2(P^t) \psi_A, \psi_A \rangle_{\pi} \\
	&= \gamma_A^2(M_2(P^t))(1-\lambda_2^{(2)}(P^t)) + \langle (I - \Pi_2(P^t)) \psi_A, \psi_A \rangle_{\pi} \\
	&= 1 - \gamma_A^2(M_2(P^t)) \lambda_2^{(2)}(P^t)\,,
	\end{align*}
	where the second equality follows from \eqref{eq:3} and the inequality follows from Cauchy-Schwartz inequality. Finally, we show \eqref{eq:2}. Using the Rayleigh quotient lower bound on the self-adjoint kernel $I - M_1(P^t)$ yields
	\begin{align*}
	\langle (I - M_1(P^t)) \psi_A, \psi_A \rangle_{\pi} &\geq 1 - \lambda_2^{(1)}(P^t)\,.
	\end{align*}
\end{proof}

\begin{proof}[Proof of Theorem \ref{thm:Cheeger}]
	We first show the upper bound of \eqref{eq:Cheeger}. We have
	$$\Phi(A) = 1 - Q(A,A) \leq 1 - \dfrac{\langle M_2\1_{A}, \1_{A} \rangle_{\pi}}{\langle \1_{A}, \1_{A} \rangle_{\pi}} = \Phi(A)(M_2) \leq O(k^4) \sqrt{1-\lambda_k^{(2)}}\,,$$
	where we apply the Peskun ordering \eqref{eq:Peskun} in the first inequality and the second inequality comes from the Cheeger's inequality for reversible chain  if $M_2$ is Markov. In the general case however, we can write $M_2 = G + I$, where $G$ is the Markov generator of $M_2$, and apply the corresponding version of Cheeger's inequality for $G$ instead (see e.g. \cite[Theorem $2$]{M15}), so the desired upper bound follows from the min-max characterization of the $k$-way expansion. Next, for the lower bound of \eqref{eq:Cheeger}, we again use the Peskun ordering \eqref{eq:Peskun} to get
	$$\Phi(A) = 1 - Q(A,A) \geq 1 - \dfrac{\langle M_1\1_{A}, \1_{A^c} \rangle_{\pi}}{\langle \1_{A}, \1_{A} \rangle_{\pi}} = \Phi(A)(M_1) \geq \dfrac{1-\lambda_k^{(1)}}{2}\,,$$
	where the second inequality follows from the Cheeger's inequality for reversible chain.
\end{proof}

\subsection{Examples}

In this section, we present an example of asymmetric random walk on $n$-cycle and a numerical example of upward skip-free Markov chain to investigate the sharpness of the spectral bounds presented in Theorem \ref{thm:mainMH} and \ref{thm:mainMH2}.

\begin{example}[Asymmetric random walk on the $n$-cycle]
	Recall that we have studied the asymmetric random walk on the $n$-cycle in Example \ref{ex:srw}, in which we adapt the notations therein. In particular, we have $l := \min\{p,q\}$ and $r := \max\{p,q\}$. For any partition $\mathcal{D} = \{A_1,A_2,\ldots,A_j\}$ with $0 \leq j-1 < n/2$, the upper bound in Theorem \ref{thm:mainMH} now gives
	\begin{align*}
	1 + \sum_{k=1}^{j-1} 1 - 2l (1-\cos(2\pi k/n)) &= j(1-2l) + 2l \sum_{k=0}^{j-1} \cos(2\pi k/n) \\
	&= j(1-2l) + 2l \left( \dfrac{\sin(\pi j/n)\cos(\pi (j-1)/n)}{\sin(\pi/n)} \right)\,.
	\end{align*}
	On the other hand, we have for $j \geq 2$,
	\begin{align*}
	\rho_j = ||\Gamma \phi_j^{(2)}||_{\pi}^2 
	= \sum_{k=1}^{j} \sum_{x \in A_k} \pi(x) \left( \dfrac{\langle \phi_j^{(2)}, \1_{A_k} \rangle_{\pi}}{\pi(A_k)} \right)^2 
	&= \sum_{k=1}^{j} \sum_{x \in A_k} \dfrac{n}{|A_k|^2}  \langle \phi_j^{(2)}, \1_{A_k} \rangle_{\pi}^2 \\
	&= \sum_{k=1}^{j} \dfrac{1}{n|A_k|} \left(\sum_{x \in A_k} \cos(2\pi (j-1)x/n)\right)^2,
	\end{align*}
	and $a = 1-2r + 2r \min_{i \geq 2} \cos(2\pi(i-1)/n)$, so the lower bound of \ref{thm:mainMH} is readily computable.
	
	Next, we now look at the leakage in Theorem \ref{thm:mainMH2} with the partition $\mathcal{D} = \{A,B\}$. The lower bound becomes
	$$	1 - \lambda_2^{(1)}(P^t) = 2l (1-\cos(2\pi/n))\,,$$
	while we note that 
	$$\langle \psi_A, \phi_2^{(2)}(P) \rangle_{\pi} = \dfrac{1}{n}  \left( \sqrt{\dfrac{|B|}{|A|}} \sum_{x \in A} \cos(2\pi x/n) - \sqrt{\dfrac{|A|}{|B|}} \sum_{x \in B} \cos(2\pi x/n)\right)\,,$$
	so the upper bound in Theorem \ref{thm:mainMH2} now reads
	$$1 - \dfrac{1}{n^2}  \left( \sqrt{\dfrac{|B|}{|A|}} \sum_{x \in A} \cos(2\pi x/n) - \sqrt{\dfrac{|A|}{|B|}} \sum_{x \in B} \cos(2\pi x/n)\right)^2 \left( 1 - 2r (1-\cos(2\pi/n)) \right)\,.$$
\end{example}

\begin{example}[Upward skip-free]
	We consider an upward skip-free chain on $\{1,2,3,4\}$ with Markov kernel given by 
	\begin{equation*}\label{eq:example}
	P = \begin{pmatrix}
	0.5 & 0.5 & 0 & 0 \\
	0.2 & 0.6 & 0.2 & 0 \\
	0.1 & 0.3 & 0.5 & 0.1 \\
	0.1 & 0.2 & 0.4 & 0.3 \\
	\end{pmatrix}
	\end{equation*}
	with eigenvalues $(1,0.52,0.25,0.13)$. The two Metropolis kernels are
	\begin{equation*}\label{eq:example2}
	M_1 = \begin{pmatrix}
	0.6 & 0.4 & 0 & 0 \\
	0.2 & 0.66 & 0.14 & 0 \\
	0 & 0.3 & 0.64 & 0.06 \\
	0 & 0 & 0.4 & 0.6 \\
	\end{pmatrix}, \quad 
	M_2 = \begin{pmatrix}
	0.40 & 0.5 & 0.09 & 0.01 \\
	0.25 & 0.54 & 0.2 & 0.01 \\
	0.1 & 0.44 & 0.36 & 0.1 \\
	0.1 & 0.2 & 0.7 & 0 \\
	\end{pmatrix}\,,
	\end{equation*}
	with eigenvalues $\lambda^{(1)} = (1,0.74,0.50,0.28)$ and $\lambda^{(2)} = (1, 0.37, 0.08, -0.16)$ respectively.
	In Theorem \ref{thm:mainMH}, we have an upper bound $1 + \lambda_2^{(1)} = 1.74$ and lower bound $1 + \rho_2 \lambda_2^{(2)} + c = 1 + 0.09\times0.37 + (-0.16)\times (1-0.09) = 0.89$.
	
	First, we consider the partition $\mathcal{D} = \{\{1,2\},\{3,4\}\}$ with $A_1 = \{1,2\}$ and $A_2 = \{3,4\}$. We see that
	$$m(\mathcal{D}) = p(A_1,A_1) + p(A_2,A_2) = 0.87 + 0.61 = 1.48\,,$$
	which is closer to the upper bound $1.74$. If we instead consider the partition $\mathcal{D} = \{\{1,2,3\},\{4\}\}$ with $A_1 = \{1,2,3\}$ and $A_2 = \{4\}$, then 
	$$m(\mathcal{D}) = p(A_1,A_1) + p(A_2,A_2) = 0.98 + 0.3 = 1.28\,,$$
	which is closer to the lower bound of $0.89$.
\end{example}
\noindent \textbf{Acknowledgements}.
The author thanks Pierre Patie, the associate editor and three referees for constructive comments that have improved the quality and presentation of the paper. This work was partially supported by NSF Grant DMS-1406599 and the Chinese University of Hong Kong, Shenzhen grant PF01001143.

\bibliographystyle{abbrvnat}
\bibliography{thesis}

\end{document}